\documentclass{article}
\usepackage{booktabs}
\usepackage[utf8]{inputenc}
\usepackage[table]{xcolor}
\usepackage{float} 

\usepackage[numbers]{natbib}
\usepackage{arxiv}
\usepackage{float}
\usepackage[utf8]{inputenc} 
\usepackage[T1]{fontenc}    
\usepackage{hyperref}       
\usepackage[table]{xcolor} 
\usepackage{multirow}
\usepackage{url}            
\usepackage{booktabs}       
\usepackage{amsfonts}       
\usepackage{nicefrac}       
\usepackage{microtype}      
\usepackage{lipsum}
\usepackage{graphicx}
\usepackage{amsmath}
\usepackage{amsthm}
\usepackage{amsthm, amsmath, amssymb}
\usepackage{hyperref}
\usepackage{mathrsfs}
\usepackage{graphicx}
\usepackage{comment}
\usepackage[normalem]{ulem}
\usepackage{booktabs}  
\usepackage{graphicx}  

\usepackage{subcaption}

\newtheorem{theorem}{Theorem}[section]
\newtheorem{lemma}[theorem]{Lemma}

\newtheorem{example}[theorem]{Example}

\theoremstyle{definition}  

\newtheorem{Definition}[theorem]{Definition}

\newtheorem{remark}{Remark}[section]

\graphicspath{ {./images/} }

\usepackage{xcolor}
\everymath{\displaystyle}

\title{Block Schwarz methods and preconditioning strategies using Generalized locally Toeplitz tools - part I: analysis of the preconditioners and numerical validation}

\author{
 A. Rifqui\\
  \\
  \texttt{abdessadek.rifqui@um6p.ma} \\
   \And
 A. Ratnani \\
 \\
  \texttt{ahmed.ratnani@um6p.ma} \\
  \And
S. Serra-Capizzano\\
 \\
  \texttt{s.serracapizzano@uninsubria.it} \\
}

\begin{document}
\maketitle
\begin{abstract}
In the current work we present a spectral analysis of the additive and multiplicative Schwarz methods within the framework of domain decomposition techniques, by investigating the spectral properties of these classical Schwarz preconditioning matrix-sequences, with emphasis on their convergence behavior and on the effect of transmission operators. In particular, after a general presentation of various options, we focus on restricted variants of the Schwarz methods aimed at improving parallel efficiency, while preserving their convergence features. In order to rigorously describe and analyze the convergence behavior, we employ the theory of generalized locally Toeplitz (GLT) sequences, which provides a robust framework for studying the asymptotic spectral distribution of the discretized operators arising from Schwarz iterations. By associating each operator sequence with the appropriate GLT symbol, we derive explicit expressions for the GLT symbols of the convergence factors, for both additive and multiplicative Schwarz methods. The GLT-based spectral approach offers a unified and systematic understanding of how the spectrum evolves with mesh refinement and overlap size (in the algebraic case). Our analysis not only deepens the theoretical understanding of classical Schwarz methods, but also establishes a foundation for examining future restricted or hybrid Schwarz variants using GLT symbolic spectral tools. Numerical experiments are presented, while, based on the study in the current work, the analysis of preconditioned matrix-sequences and proposals of new Schwarz preconditioners are given in a twin paper, ideally part II of the present work.

\end{abstract}


\section{Introduction}

Standard finite difference and finite element discretizations of partial differential equations (PDEs) posed on $d$-dimensional domains, with $d\ge 1$, typically yield $d$-level banded matrices, including $d$-level tridiagonal and pentadiagonal structures. In Isogeometric Analysis (IgA)~\cite{hughes2005}, the resulting matrices retain a $d$-level structure, but their bandwidth generally grows as a consequence of the higher continuity of the spline basis functions. The associated linear systems are therefore large, highly structured, and spectrally nontrivial. In particular, as the spline degree $p$ increases under $p$-refinement, the number of nonzero entries per row grows at each of the $d$ levels, resulting in a substantial increase in the computational cost of the corresponding linear algebra. The design of robust and scalable solvers for high-degree spline discretizations is accordingly a central issue, and has motivated a substantial body of work on efficient iterative methods~\cite{Donatelli2017, Buffa2017, Sangalli2016, Collieretall,Takacs2015}. Analogous considerations apply to high-order finite element discretizations, where, although the bandwidth may remain fixed, scalar entries are replaced by blocks whose size depends on the polynomial degree and on the spatial dimension $d$, leading to a comparable increase in structural complexity~\cite{GSS,RymaNLAA}.

Schwarz methods form a fundamental class of domain decomposition techniques for the numerical solution of PDEs. Their underlying principle is to partition the computational domain into subdomains and to solve local problems, either exactly or approximately, on each subdomain. In the classical Schwarz method, Dirichlet conditions are imposed on the artificial interfaces separating neighboring subdomains~\cite{Mathew2008,Smith1996,Toselli2005,quarteroni1999domain}. Algebraic interpretations of Schwarz methods have played a significant role in clarifying their structure and in motivating several extensions~\cite{Tang1992,Griebel1995,Frommer1999,Benzi2001}. From an algebraic standpoint, in particular, additive and multiplicative Schwarz methods are closely related to block Jacobi and block Gauss--Seidel iterations, respectively, with the additional feature of overlap between subdomains.

The convergence analysis of the corresponding iterative methods nonetheless presents several difficulties. First, the matrices involved are typically of large size. More significantly, robustness with respect to the problem and discretization parameters requires a parametric analysis, at least with respect to the discretization parameter $h=h_n$ (or, in the multilevel setting, the corresponding mesh parameters), where $h_n\to0$ as increasingly accurate approximations are sought. From a matrix-theoretic perspective, the limit $h_n\to0$ corresponds to the matrix dimension $d_n$ tending to infinity. It is therefore natural to consider not a single matrix $A_n$ but the entire matrix-sequence $\{A_n\}_n$, where $A_n$ has size $d_n\times d_n$ and $d_n\to\infty$ as $n\to\infty$. This places the analysis naturally within the framework of asymptotic linear algebra~\cite{BG-book,G.Barbarino}, with related convergence results of (preconditioned) conjugate gradients based on asymptotic spectral distribution via potential theory \cite{superlinear-potential1,superlinear-potential2}.

For PDE discretizations with constant coefficients on Cartesian domains, uniform meshes, and Dirichlet boundary conditions, one typically obtains $d$-level Toeplitz structures. Block structures may likewise arise, for instance, in high-order finite elements, Isogeometric Analysis with intermediate regularity, and discontinuous Galerkin methods; see~\cite{Barbarino2020,Barbarino2020-bis} and the references therein. These settings, however, are special cases: violating any of the assumptions above generally destroys the global Toeplitz structure of the coefficient matrices. In particular, changes in boundary conditions can often be interpreted as asymptotically low-rank perturbations of the underlying Toeplitz structure, whereas variable coefficients, nonuniform meshes, or more general geometries typically give rise to matrix-sequences that are globally non-Toeplitz but retain a hidden local Toeplitz structure.

The theory of generalized locally Toeplitz (GLT) sequences provides a unified framework for describing and analyzing the spectral and singular value distributions of matrix-sequences exhibiting such structured, possibly hidden, Toeplitz behavior. This framework is particularly well suited to matrix-sequences arising from the discretization of PDEs under general assumptions; see~\cite{barbarino2022systematic,Barbarino2020,Barbarino2020-bis,garoni2017generalized,garoni2018generalized} for the general theory and~\cite{garoni2025introduction,glt-tutorial,garoni2019multilevel} for application-oriented surveys. GLT sequences arise naturally in the discretization of both differential and integral equations~\cite{SerraCapizzano2006,garoni2017generalized,garoni2018generalized}, and their analysis has established close connections with function spaces, spectral approximation, and operator theory~\cite{G.Barbarino}. The spectral and singular value distributions associated with a GLT sequence, in particular, provide useful information for several purposes, including (i) assessing the spectral accuracy of discretization schemes~\cite{bianchi2021spectral}, (ii) verifying the preservation of average spectral gaps~\cite{BianchiSerraCapizzano2018}, and (iii) predicting eigenvalue distributions analytically~\cite{garoni2018generalized}. Moreover, GLT analysis can be used to derive precise information on the asymptotic behavior of iterative solvers, including Krylov methods~\cite{SerraCapizzano2006,garoni2017generalized,garoni2018generalized}.

A further advantage of the GLT framework is that it can be employed constructively in the design and analysis of preconditioners, multigrid operators, and multi-iterative methods; see \cite{glt-ratnani} and references therein. A representative example is the construction of multigrid and multi-iterative solvers for high-degree Isogeometric Analysis \cite{Donatelli2017}; see also~\cite{Takacs2015} for a complementary approach based on suitably designed approximation spaces.

\ \
\noindent
\subsubsection*{Main Contributions}
\noindent

Let ${A_n}$ denote a matrix-sequence arising from the discretization of an elliptic problem or, more generally, a block-structured matrix-sequence admitting a GLT symbol $\kappa(\mathbf{x},\boldsymbol{\theta})$. Within the GLT framework, and from a purely algebraic viewpoint, we establish the following results.

\begin{itemize}
	\item In the nonoverlapping case, both the additive and multiplicative Schwarz operators generate matrix-sequences ${P_n}$ satisfying
	$$
	{P_n}\sim_{\mathrm{GLT}}\kappa(\mathbf{x},\boldsymbol{\theta}).
	$$
	
	\item In the overlapping case, the multiplicative Schwarz operator and its restricted variants satisfy the same GLT relation. The corresponding property does not, however, hold in general for the additive Schwarz operator; in that case, a restriction is required to guarantee
	$$
	{P_n}\sim_{\mathrm{GLT}}\kappa(\mathbf{x},\boldsymbol{\theta}).
	$$
\end{itemize}

Preservation of the GLT symbol does not, however, preclude the presence of outlier eigenvalues or singular values. Such outliers may substantially degrade the effectiveness of $P_n$ as a preconditioner, particularly when they approach zero. We address this issue by proposing computationally inexpensive strategies for controlling their effect. In this respect, the present work constitutes a first step toward several of the research directions suggested in~\cite{GLH}.

In the companion paper~\cite{twin}, with the exception of the overlapping additive Schwarz case, we further show that the preconditioned matrix-sequences associated with the methods considered here satisfy
$$
{P_n^{-1}A_n}\sim_{\mathrm{GLT}}1,
$$
which provides a rigorous spectral-distribution justification for the corresponding preconditioned iterative methods. Furthermore, the GLT framework yields constructive information that can be exploited to devise new variants of the methods considered here while preserving the fundamental property
${P_n^{-1}A_n}\sim_{\mathrm{GLT}}1$, which, for a symbol $\kappa(\mathbf{x},\boldsymbol{\theta})$ sparsely vanishing i.e. with zeros of zero measure, reduces to ${P_n}\sim_{\mathrm{GLT}}\kappa(\mathbf{x},\boldsymbol{\theta})$ whenever ${A_n}\sim_{\mathrm{GLT}}\kappa(\mathbf{x},\boldsymbol{\theta})$.

\subsubsection*{Organization of the Work}

In Section \ref{sec2:block prec} we introduce block preconditioners which are used in the context of domain decomposition techniques.
Section~\ref{sec3:glt th} is devoted to introduce the axioms that characterize the GLT theory and other tools recently developed for block structures.

Section~\ref{sec5:unified glt} presents our main results, namely a unified GLT-based analysis of Schwarz block preconditioners for GLT sequences. We prove that the block Jacobi, block Gauss--Seidel, restricted additive Schwarz, multiplicative Schwarz, and restricted multiplicative Schwarz preconditioners all preserve the GLT symbol of the underlying sequence. In contrast, the additive Schwarz preconditioner fails to preserve the GLT symbol in the overlapping case, and restriction is shown to be essential for recovering this property. Consequently, whenever the GLT symbol is preserved, the preconditioned matrices exhibit spectral clustering at $1$, which in turn guarantees rapid convergence of the associated Krylov methods. These results provide the first comprehensive GLT-theoretic justification for the effectiveness of Schwarz-type preconditioners.

In Section~\ref{sec6:num}, we present numerical experiments that validate our theoretical predictions, showing that the proposed Schwarz operators perform effectively both as standalone iterative methods and as preconditioners within Krylov solvers. Section~\ref{sec7:end} presents concluding remarks.

\section{Iterative methods with classical block preconditioners}\label{sec2:block prec}

The main objective of this study is to design and analyze an optimal solver
based on the domain decomposition method for the linear system
\begin{equation}
	A_n u = F,
	\label{eq:linear_system}
\end{equation}
where $A_n$ is a square matrix of size $d_n \times d_n$.

In general, the global matrix $A_n$ exhibits a block-tridiagonal structure
associated with a decomposition of the computational domain $\Omega$ into
$\nu$ overlapping or non-overlapping subdomains $\Omega_1,\ldots,\Omega_\nu$,
$\Omega=\cup_{i=1}^\nu \Omega_i$. In the context of the present work we
consider mainly $\Omega=[0,1]$, $\Omega_i=[\alpha_i,\beta_i]$,
$\alpha_i<\beta_i$, $i=1,\ldots,\nu$.

The latter decomposition leads to a block-structured linear system of the form
\begin{equation}\label{eq:general_matrix_nu}
	A_n =
	\begin{bmatrix}
		A_{11} & A_{12} & O      & \cdots & O \\
		A_{21} & A_{22} & A_{23} & \ddots & \vdots \\
		O      & A_{32} & A_{33} & \ddots & O \\
		\vdots & \ddots & \ddots & \ddots & A_{3(\nu-1),3\nu-2} \\
		O      & \cdots & O & A_{3\nu-2,3(\nu-1)} & A_{3\nu-2,3\nu-2}
	\end{bmatrix},
\end{equation}
where each block $A_{ij}$ is of dimension ${n_i} \times {n_j}$, and the
total size satisfies
\[
d_n = \sum_{i=1}^{3\nu-2} {n_i}.
\]
The diagonal blocks $A_{ii}$ correspond to the discrete operators on each
subdomain $\Omega_i$, while the off-diagonal blocks $A_{i,i+1}$ and
$A_{i+1,i}$ represent the coupling conditions between consecutive subdomains; see Section 5.1 of \cite{GanderZhang2019} for a
closely related block-tridiagonal structure arising from sequential domain
decompositions.

\begin{remark}[\textbf{Geometric description of overlap}]\label{rem:non and overlap with nu subd}
	By recalling the setting at the beginning of Section~\ref{sec2:block prec},
	the general case of $\nu$ non-overlapping subdomains
	$\Omega_1,\ldots,\Omega_\nu$, $[0,1]=\Omega=\cup_{i=1}^\nu \Omega_i$,
	$\Omega_i=[\alpha_i,\beta_i]$, $\alpha_i<\beta_i$, $i=1,\ldots,\nu$, is
	described by the conditions $\beta_i=\alpha_{i+1}$, $i=1,\ldots,\nu-1$.
	
	The general case of $\nu$ overlapping subdomains is described by the
	conditions $\alpha_i<\alpha_{i+1}$ and $\beta_i<\beta_{i+1}$,
	$i=1,\ldots,\nu-1$ (this ensures that a subdomain is not a subset of
	another), together with the existence of an index $j$ such that
	$\alpha_{j+1}<\beta_j$. The set of indices $j\in\{1,\ldots,\nu-1\}$ such
	that $\alpha_{j+1}<\beta_j$ is denoted by $\mathcal{OV}$. Furthermore, we
	require $\beta_i<\alpha_{i+2}$, $i=1,\ldots,\nu-2$; otherwise the subdomain
	$\Omega_{i+1}\subset\Omega_i\cup\Omega_{i+2}$ would be redundant.
	
	As a consequence, the presence of overlapping is equivalent to the condition $\#\mathcal{OV}\ge 1$ or, in other terms,
	$\#\mathcal{OV}=0$ if and only if no overlapping is observed.
	
	Furthermore, by $a_i$ we indicate the characteristic function on $[0,1]$ of
	the domain $\Omega_i=[\alpha_i,\beta_i]$, $i=1,\ldots,\nu$. In the
	overlapping case, further sets have to be defined. More precisely, we introduce the sets $S_i,T_i$ with
	$S_i\cup T_i=\Omega_i\cap\Omega_{i+1}$ and $S_i\cap T_i=\emptyset$, where
	$\Omega_i$, $S_i$, and $T_i$ contain, respectively, $d_i(n)+n_{3i-4}+n_{3i}$,
	$n_{3i-1}$, and $n_{3i+1}$ grid points, $i=1,\ldots,\nu-1$, with the convention
	that $n_j=0$ when $j\le 0$ or $j\ge 3\nu+1$.
\end{remark}

\begin{itemize}
	\item We first consider the following \textbf{diagonal block} (without
	overlap, $\#\mathcal{OV}=0$). For $i=1,\ldots,\nu$ let $d_i(n)$ be the
	number of degrees of freedom attached to $\Omega_i$, so that
	\begin{equation}\label{eq:dn-general}
		d_n = \sum_{i=1}^{\nu} d_i(n).
	\end{equation}
	Define the restriction operators
	\begin{equation}\label{eq:Ri-general-noov}
		R_i = \big[\, O \ \cdots\ O \ \ I_{d_i(n)} \ \ O\ \cdots\ O \,\big],
		\qquad i=1,\ldots,\nu,
	\end{equation}
	where the identity block occupies the columns of $A_n$ associated with
	$\Omega_i$; by construction $\sum_{i=1}^{\nu} R_i^T R_i = I_{d_n}$, i.e.
	$\{R_i\}_{i=1}^\nu$ realizes a genuine partition of unity (no double
	counting). For each $i=1,\ldots,\nu$, the transpose $R_i^T$ acts as a
	\textbf{prolongation operator}; the terminology is standard in the
	multigrid literature.
	
	The local operators are
	\begin{equation}\label{eq:Ai-general-noov}
		A_i := R_i A_n R_i^T, \qquad i=1,\ldots,\nu,
	\end{equation}
	square of size $d_i(n)$, and
	\begin{equation}\label{no overlap nu-general}
		\mathrm{BlockDiag}_{d_1(n),\ldots,d_\nu(n)}(A_n) =
		\begin{bmatrix}
			A_{1} & O & \cdots & O \\
			O     & A_{2} & \cdots & O \\
			\vdots & \vdots & \ddots & \vdots \\
			O & O & \cdots & A_{\nu}
		\end{bmatrix}.
	\end{equation}
	The operators $\mathrm{BlockDiag}(\cdot)$ are studied in some detail in
	\cite{GLH} in connection with the new class of generalized locally
	Hankel matrix-sequences.
	
	The corresponding \textbf{block Jacobi preconditioner} $P_n$ is that
	defined in \eqref{no overlap nu-general}, and we have
	\begin{equation}\label{BJ:block-nu}
		P^{-1}_n = \left(P_{BJ, n}^{(\nu)}(A_n) \right)^{-1}:=
		\left(\mathrm{BlockDiag}_{d_1(n),\ldots,d_\nu(n)}(A_n)\right)^{-1}
		= \sum_{i = 1}^{\nu} R_i^{T} A_{i}^{-1} R_i.
	\end{equation}
	
	The standard \textbf{block Jacobi iteration} using these $\nu$ blocks
	has the iteration operator
	\begin{equation}\label{eq:TBJ-nu}
		T_n = T_{BJ, n}^{(\nu)}(A_n) := I - \left(P_{BJ, n}^{(\nu)}\right)^{-1} A_n
		= I - \sum_{i=1}^{\nu} R_i^T A_i^{-1} R_i A_n.
	\end{equation}
	For a given initial vector $u_0$, the iterative scheme is expressed as
	\[
	u^{k+1} = T_n u^k + P^{-1}_nF, \qquad k = 0, 1, 2, \dots,
	\]
	and converges linearly with an asymptotic convergence factor
	$\rho(T_n)$, the spectral radius of the iteration operator. Later, we
	will analyze the spectral features of the matrix-sequence $\{T_n\}$ in
	detail, using the GLT theory.
	
	\item Similarly, using the $\mathrm{BlockTril}(\cdot)$ operators studied
	in \cite{GLH}, the \textbf{block Gauss--Seidel preconditioner} (without
	overlap, $\#\mathcal{OV}=0$) is defined as
	\begin{equation}\label{BGS:block-nu}
		P^{-1}_n = \left(P_{BGS, n}^{(\nu)}(A_n)\right)^{-1}:=
		\left(\mathrm{BlockTril}_{d_1(n),\ldots,d_\nu(n)}(A_n)\right)^{-1}
		=
		\left( I - \prod_{i=\nu}^{1} \big(I - R_i^{T} A_{i}^{-1} R_i A_n \big) \right) A_n^{-1},
	\end{equation}
	where the blocks $A_i$ and the restriction operators $R_i$,
	$i=1,\ldots,\nu$, are the same as in the previous block Jacobi
	definition, and $\mathrm{BlockTril}(\cdot)$ retains the
	block-lower-triangular part of $A_n$ (including the diagonal blocks
	$A_i$) with respect to the partition \eqref{eq:dn-general}. Again, the
	standard \textbf{block Gauss--Seidel iteration} using these $\nu$
	blocks has the iteration operator
	\begin{equation}\label{eq:TBGS-nu}
		T_n = T_{BGS, n}^{(\nu)} (A_n):= I - \left(P_{BGS, n}^{(\nu)}\right)^{-1} A_n
		= \prod_{i=\nu}^{1} \big(I - R_i^{T} A_{i}^{-1} R_i A_n\big),
	\end{equation}
	the product being ordered from $i=\nu$ down to $i=1$.
	
	\item Consider now the same construction \emph{with overlap}, i.e.
	$\#\mathcal{OV}\ge 1$. For each $i=1,\ldots,\nu$ let $R_i$ now be the
	restriction operator selecting the degrees of freedom associated with
	the (enlarged, overlapping) subdomain $\Omega_i$, of size
	$\left(\sum_{k\in K_i} n_k\right)\times d_n$, where $K_i$ is the index
	set of blocks of $A_n$ overlapping $\Omega_i$. For an interior
	subdomain overlapping both neighbors this takes the form
	\begin{equation}\label{block Schwarz-nu}
		A_i =
		\begin{bmatrix}
			A_{i-1,i-1} & A_{i-1,i} & O \\
			A_{i,i-1} & A_{i,i} & A_{i,i+1} \\
			O      & A_{i+1,i} & A_{i+1,i+1}
		\end{bmatrix}, \qquad A_i := R_i A_n R_i^T,
	\end{equation}
	consistent with the row/column blocks of $A_n$ in
	\eqref{eq:general_matrix_nu}, with the obvious modification at $i=1$
	and $i=\nu$ (which involve only one neighbor). Unlike the
	non-overlapping case, $\sum_{i=1}^\nu R_i^T R_i \ge I_{d_n}$
	componentwise, with equality if and only if $\#\mathcal{OV}=0$.
	
	Using this notation, the \textbf{block additive} and \textbf{block
		multiplicative Schwarz preconditioners} are defined as
	\begin{equation}\label{eq:add_mul_pre-nu}
		\left(P_{BAS, n}^{(\nu)} (A_n)\right)^{-1}:= \sum_{i=1}^{\nu} R_i^{T} A_{i}^{-1} R_i,
		\qquad
		\left(P_{BMS, n}^{(\nu)}(A_n)\right)^{-1} :=
		\left( I - \prod_{i=\nu}^{1} \big(I - R_i^{T} A_{i}^{-1} R_i A_n \big) \right) A_n^{-1},
	\end{equation}
	respectively.
	
	\begin{remark}\label{rem:no overlap ie BJ-BGS}
		Comparing these blocks, one concludes that the classical Schwarz
		preconditioners can be regarded as a generalization of
		\textbf{block Jacobi} or \textbf{block Gauss--Seidel}
		preconditioners, with the addition of overlap. Indeed, when
		considering the theoretical analysis, we show how to interpret the
		overlapping case in terms of a combination of several
		\textbf{block Jacobi} or \textbf{block Gauss--Seidel}
		preconditioners, respectively, acting simultaneously: this is the
		key for using the GLT theoretical machinery in our setting.
	\end{remark}
	
	\item From the explicit expression of the preconditioners
	\eqref{eq:add_mul_pre-nu}, one can explicitly write the iteration
	operators for the additive and multiplicative Schwarz methods as
	\begin{align}\label{add_mul_Opr-nu}
		T_{\text{BAS}, n}^{(\nu)}(A_n) &= I - \sum_{i=1}^{\nu} R_i^T A_i^{-1} R_i A_n, \\
		\label{mul_pre_only-nu}
		T_{\text{BMS}, n}^{(\nu)} (A_n) &= \prod_{i=\nu}^{1} \left(I - R_i^T A_i^{-1} R_i A_n\right),
	\end{align}
	respectively.
	
	The additive Schwarz iteration (with overlap) associated with the
	operator $T_{\text{BAS}, n}^{(\nu)}$ is usually \textbf{not
		convergent}, since for overlapping subdomains it holds that
	$2I \ge R_i^T R_i \ge I$, with constants that cannot be improved. To
	address this, one commonly introduces a relaxation parameter
	$0 < \gamma < 1$. However, an alternative approach is provided by the
	\textbf{restricted additive Schwarz (RAS)} method \cite{GanderRas,
		CaiSarkis1999, frommer2001algebraic}. In the theoretical section we
	will show that the GLT theory explains why one should expect RAS to be
	\emph{faster} than the additive Schwarz method with relaxation
	parameter $\gamma$.
	
	The RAS method uses the local solvers with overlap, employing the
	standard restriction operators $R_i$, but uses \emph{prolongation
		operators without overlap} $\widetilde{R}_i^T$: $\widetilde{R}_i$ is
	the restriction operator $R_i$ with the overlapping rows zeroed out, so
	that each degree of freedom is selected exactly once,
	\begin{equation}\label{eq:partition-unity-nu}
		\sum_{i=1}^{\nu} \widetilde R_i^T R_i = I_{d_n},
	\end{equation}
	eliminating double counting of variables in the overlap. Under certain
	assumptions, there is no need for a relaxation parameter to ensure
	convergence. Set
	\begin{equation}\label{eq:hatR-nu}
		\widehat R_i = \begin{cases}
			\widetilde R_i, & i\in \mathcal{OV} \text{ (or adjacent to an overlap)},\\
			R_i, & \text{otherwise},
		\end{cases}
		\qquad i=1,\ldots,\nu,
	\end{equation}
	so that $\widehat R_i = R_i$, whenever $\#\mathcal{OV}=0$. With this
	notation, the RAS iteration operator is
	\begin{equation}\label{eq:TRAS-nu}
		T_{\text{BRAS}, n}^{(\nu)}(A_n) = I - \sum_{i=1}^{\nu}  \widehat{R}_i^T A_i^{-1} R_i A_n.
	\end{equation}
	
	Similarly, the \textbf{restricted multiplicative Schwarz (RMS)}
	\cite{nabben2002convergence} iteration operator is defined as
	\begin{equation}\label{eq:TRMS-nu}
		T_{\text{BRMS}, n}^{(\nu)}(A_n) = \prod_{i=\nu}^{1} \left(I -  \widehat{R}_i^T A_i^{-1} R_i A_n\right).
	\end{equation}
	When $\#\mathcal{OV}=0$, \eqref{eq:TRAS-nu} and \eqref{eq:TRMS-nu}
	coincide with the block Jacobi operator \eqref{eq:TBJ-nu} and the
	block Gauss--Seidel operator \eqref{eq:TBGS-nu}, respectively, since
	$\widehat R_i = R_i$ for every $i$.
\end{itemize}

\begin{remark}\label{rem:no need restriction in multiplicative}
	Although for the multiplicative Schwarz the use of $\widehat{R}_i^T$
	is not strictly necessary to avoid double counting, this method is
	included here for completeness.
\end{remark}

\begin{remark}\label{rem:symmetric restrictions}
	When a real symmetric (or Hermitian) problem is considered, it is a
	good numerical choice to preserve the real symmetry (or the Hermitian
	character) in the design of a preconditioner. In this spirit, the
	symmetrized version of the RAS iteration matrices is defined as
	\begin{equation}\label{eq:TRASsym-nu}
		T_{\text{BRAS-sym}, n}^{(\nu)}(A_n) = I - \sum_{i=1}^{\nu}  \widehat{R}_i^T A_i^{-1} \widehat{R}_i A_n,
	\end{equation}
	with $\widehat{R}_i=D_i{R}_i$ such that $\sum_{i=1}^{\nu}
	\widehat{R}_i^T{R_i}$ is the identity and $D_i$ are diagonal.
	Unfortunately, considering the corresponding expression
	$\prod_{i=\nu}^{1} \left(I -  \widehat{R}_i^T A_i^{-1} \widehat{R}_i A_n\right)$,
	we obtain a variation of the RMS iteration matrices which is still
	non-symmetric, while the formulation
	\begin{equation}\label{eq:TRMSsym-nu}
		T_{\text{BRMS-sym}, n}^{(\nu)}(A_n) = T^TT, \qquad T=T_{\text{BRMS}, n}^{(\nu)}(A_n)
	\end{equation}
	is.
	
	In all the previous reasoning, in order to apply the GLT theory, the
	diagonal weight matrices $D_i$, $i=1,\ldots,\nu$, are chosen as
	diagonal sampling matrices.
\end{remark}

\section{Overview of the Theory of GLT Sequences}\label{sec3:glt th}

In the present section, we provide a concise overview of the main tool that we use in this paper, i.e., the theory of GLT matrix-sequences, which is instrumental in analyzing the spectral distribution of additive and multiplicative Schwarz methods and their variants, including the restricted versions. For brevity, we present only the results that are directly needed in this work. For readers new to the subject, we refer to \cite{glt-tutorial,garoni2019multilevel,garoni2025introduction}, while more advanced treatments can be found in \cite{barbarino2022systematic,Barbarino2022, Barbarino2020, garoni2017generalized, garoni2018generalized}.\\
\textbf{Singular Value and Spectral Distribution of a Sequence of Matrices.}
Let \(\mu_k\) denote the Lebesgue measure on \(\mathbb{R}^k\). Throughout this paper, standard terminology from measure theory (such as "measurable set", "measurable function", "{almost everywhere} (a.e)", etc.) always refers to the Lebesgue measure. A matrix-valued function
\[
f : D \subset \mathbb{R}^k \longrightarrow \mathbb{C}^{s \times s}
\]
is said to be measurable (respectively, continuous, continuous a.e., in \(\mathbb{L}^{1}(D)\), etc.) if each of its components
\[
f_{ij} : D \longrightarrow \mathbb{C}, \quad i,j = 1, \dots, s,
\]
is measurable (respectively, continuous, continuous a.e., in \(\mathbb{L}^{1}(D)\), etc.). For every \(x, y \in \mathbb{R}\), define \(x \wedge y = \min(x, y)\). We denote by \(C_c(\mathbb{C})\) (respectively, \(C_c(\mathbb{R})\)) the space of complex-valued (respectively, real-valued) functions with compact support. The singular values of a matrix \(A \in \mathbb{C}^{m \times n}\) are denoted by \(\sigma_i(A)\), \(i = 1, \ldots, m \wedge n\), and the eigenvalues of a matrix \(A \in \mathbb{C}^{m \times m}\) are denoted by \(\lambda_i(A)\), \(i = 1, \dots, m\). The spectrum of a matrix \(A \in \mathbb{C}^{m \times m}\) is denoted by \(\Lambda(A)\). By definition, a sequence of matrices is a sequence of the form \(\{A_n\}_n\), where \(n\) varies over an infinite subset of \(\mathbb{N}\), and each \(A_n\) is of size \(d_n \times e_n\) with both \(d_n\) and \(e_n\) tending to \(\infty\) as \(n \to \infty\).

\begin{lemma}\label{set-bar}
Let $\{A_n\}_n$ be a sequence of square matrices and let $f : D \subset \mathbb{R}^k \to \mathbb{C}^{s \times s}$ be measurable with $0 < \mu_k(D) < \infty$. If $\{A_n\}_n \sim_\lambda f$ and $\Lambda(A_n)$ is contained in $S \subset \mathbb{C}$ for all $n$, then
$
\Lambda(f) \subset \bar{S} \; \text{a.e.}
$
\end{lemma}

\begin{proof}
{See Lemma 2.3 and Corollary 2.15 in \cite{Barbarino2020}.}
\end{proof}

\begin{Definition}[\textbf{Singular value and spectral distribution of a sequence of matrices\label{def:singular_spectral_distribution}}]
	\leavevmode
	\begin{itemize}
		\item Let $\{A_n\}_n$ be a sequence of matrices with $A_n$ of size $d_n \times e_n$, and let
		$f : D \subset \mathbb{R}^k \to \mathbb{C}^{s \times t}$ be measurable with
		$0 < \mu_k(D) < \infty$.
		
		We say that $\{A_n\}_n$ has a \emph{singular value distribution} described by $f$, and we write
		\[
		\{A_n\}_n \sim_{\sigma} f,
		\]
		if
		\[
		\lim_{n \to \infty}
		\frac{1}{d_n \wedge e_n}
		\sum_{i=1}^{d_n \wedge e_n}
		F\big(\sigma_i(A_n)\big)
		=
		\frac{1}{\mu_k(D)}
		\int_D
		\left(
		\frac{1}{s \wedge t}
		\sum_{i=1}^{s \wedge t}
		F\big(\sigma_i(f(x))\big)
		\right)
		dx,
		\quad \forall\, F \in C_c(\mathbb{R}).
		\]
		
		\item Let $\{A_n\}_n$ be a sequence of square matrices with $A_n$ of size $d_n \times d_n$, and let
		$f : D \subset \mathbb{R}^k \to \mathbb{C}^{s \times s}$ be measurable with
		$0 < \mu_k(D) < \infty$.
		
		We say that $\{A_n\}_n$ has a \emph{spectral (or eigenvalue) distribution} described by $f$, and we write
		\[
		\{A_n\}_n \sim_{\lambda} f,
		\]
		if
		\[
		\lim_{n \to \infty}
		\frac{1}{d_n}
		\sum_{i=1}^{d_n}
		F\big(\lambda_i(A_n)\big)
		=
		\frac{1}{\mu_k(D)}
		\int_D
		\left(
		\frac{1}{s}
		\sum_{i=1}^{s}
		F\big(\lambda_i(f(x))\big)
		\right)
		dx,
		\quad \forall\, F \in C_c(\mathbb{C}).
		\]
	\end{itemize}
\end{Definition}
For the well-posedness of \eqref{def:singular_spectral_distribution}, we require the following lemma.

\begin{lemma} [{\cite{Barbarino2020}, Lemma 2.1}]
	Let $f : D \subseteq \mathbb{R}^k \to \mathbb{C}^{s \times s}$ be a measurable function, and
	let $g : \mathbb{C}^s \to \mathbb{C}$ be continuous and symmetric in its $s$ arguments,
	that is,
	\[
	g(\lambda_1, \ldots, \lambda_s) = g(\lambda_{\rho(1)}, \ldots, \lambda_{\rho(s)})
	\quad \text{for all permutations } \rho \text{ of } \{1, \ldots, r\}.
	\]
	Then, the mapping
	\[
	x \mapsto g\big(\lambda_1(f(x)), \ldots, \lambda_s(f(x))\big)
	\]
	is well-defined (independently of the ordering of the eigenvalues of $f(x)$) and measurable.
	
	As a consequence:
	\begin{itemize}
		\item the mapping $x \mapsto g\big(\sigma_1(f(x)), \ldots, \sigma_s(f(x))\big)$ is measurable;
		\item for every continuous function $F : \mathbb{C} \to \mathbb{C}$, the mappings
		\[
		x \mapsto \sum_{i=1}^s F\big(\lambda_i(f(x))\big)
		\quad \text{and} \quad
		x \mapsto \sum_{i=1}^s F\big(\sigma_i(f(x))\big)
		\]
		are measurable;
		\item for all $p \in [1, \infty]$, the mapping $x \mapsto \|f(x)\|_p$ is measurable.
	\end{itemize}
\end{lemma}
\textbf{Informal meaning of the singular value and eigenvalue distribution~\eqref{def:singular_spectral_distribution}.}
Assume that $f$ admits $s$ eigenvalue functions $\lambda_i(f(x))$, $i = 1, \ldots, s$,
which are continuous almost everywhere on $D$.
Then, the eigenvalues of $A_n$, except possibly for $o(d_n)$ outliers,
can be partitioned into $s$ subsets of approximately equal cardinality.
For sufficiently large $n$, the eigenvalues in the $i$-th subset
are approximately equal to the samples of the $i$-th eigenvalue function
$\lambda_i(f(x))$ evaluated on a uniform grid over the domain $D$, $i=1,\ldots,s$.

For example:
\begin{itemize}
	\item If $k = 1$, $d_n = n s$, $f$ is continuous, $D = [a, b]$, and assuming there are no outliers, then
	the eigenvalues of $A_n$ are approximately given by
	\[
	\lambda_i\!\left(f\!\left(a + \frac{j(b-a)}{n}\right)\right),
	\qquad j = 1, \ldots, n, \quad i = 1, \ldots, s,
	\]
	for sufficiently large $n$.
	
	\item If $k = 2$, $d_n = n^2 s$, $f$ is continuous, $D = [a_1, b_1] \times [a_2, b_2]$,
	and assuming there are no outliers, then the eigenvalues of $A_n$ are approximately given by
	\[
	\lambda_i\!\left(
	f\!\left(
	a_1 + \frac{j_1(b_1 - a_1)}{n},\,
	a_2 + \frac{j_2(b_2 - a_2)}{n}
	\right)\right),
	\qquad j_1, j_2 = 1, \ldots, n, \quad i = 1, \ldots, s,
	\]
	for sufficiently large $n$.
\end{itemize}

The same interpretation extends naturally to higher dimensions, $k \ge 3$, and to more general domains $D$ with $\mu_k(D)\in (0,\infty)$: of course for more general domains $D$ we have to compare with uniform samplings and hence we necessarily need that $D$ is measurable according to Peano-Jordan i.e. its characteristic function is Riemann integrable; see e.g. \cite[pp. 398--399]{SerraCapizzano2003}.
An entirely analogous reasoning applies to the singular value distribution as well.

\textbf{Zero-distributed sequences.}
A \emph{zero-distributed sequence} is a sequence of matrices $\{Z_n\}_n$ such that
$\{Z_n\}_n \sim_\sigma 0$, i.e.,
\[
\lim_{n \to \infty} \frac{1}{d_n \wedge e_n}
\sum_{i=1}^{d_n \wedge e_n} F\big(\sigma_i(Z_n)\big) = F(0),
\quad \forall F \in C_c(\mathbb{R}),
\]
where $d_n \times e_n$ is the size of $Z_n$.

\textbf{Toeplitz sequences.}
Let $f : [-\pi, \pi] \to \mathbb{C}^{s \times s}$ be a function in $L^1([-\pi, \pi])$.
Its Fourier coefficients $\{f_k\}_{k \in \mathbb{Z}}$ are defined by
\begin{equation}\label{eq:fourier_coeff}
	f_k = \frac{1}{2\pi} \int_{-\pi}^{\pi} f(\theta) e^{-ik\theta} \, d\theta \in \mathbb{C}^{s \times s},
	\quad k \in \mathbb{Z},
\end{equation}
where the integrals are computed componentwise.

The $n$-th (block) Toeplitz matrix generated by $f$ is the $ns \times ns$ matrix
$T_n(f)$ defined as
\[
T_n(f) = [f_{i-j}]_{i,j=1}^n.
\]

Any sequence of matrices of the form $\{T_{d_n}(f)\}_n$, with $d_n \to \infty$ and
$f : [-\pi, \pi] \to \mathbb{C}^{s \times s}$ in $L^1([-\pi, \pi])$, is referred to as
a \emph{(block) Toeplitz sequence} generated by $f$.

\textbf{Sequences of diagonal sampling matrices.}
Let $a : [0, 1] \to \mathbb{C}^{s \times s}$.
The $n$-th (block) diagonal sampling matrix generated by $a$ is the $ns \times ns$
(block) diagonal matrix $D_n(a)$ defined as
\[
D_n(a) = \operatorname{diag}\big(\{a(i/n)\}_{i=1}^n\big).
\]
Any sequence of matrices of the form $\{D_{d_n}(a)\}_n$, with $d_n \to \infty$ and
$a : [0, 1] \to \mathbb{C}^{s \times s}$, is referred to as a \emph{sequence of (block) diagonal sampling matrices generated by $a$}.

The construction of the GLT sequences is based on the notion of approximating class of sequences. Recently, this notion has been revisited in order to handle in a natural way also the case of approximations of PDEs on either moving or unbounded domains; see \cite{gacs}.

\textbf{GLT Sequences.} In the current section, we collect the properties of GLT sequences that are needed in this paper.
 Throughout this paper, we denote by $I_s$ the $s \times s$ identity matrix.

Before formulating these properties, we introduce some notation and terminology.

Throughout this paper, we denote by $O_{s,t}$ the $s \times s$ zero matrix. The matrix $O_{s,s}$ is denoted by $O_s$.
If the size is clear from the context, we often write $O$ instead of $O_{s,t}$ or $O_s$. If $A$ is a matrix, we denote by $A^\dagger$ the Moore--Penrose pseudoinverse of $A$.
For our purposes, it is relevant that $A^\dagger = A^{-1}$ whenever $A$ is invertible. For more details on the pseudoinverse, see \cite{Golub2013}. {When $A$ is singular, as in \cite{frommer_nabben_szyld_2008,nabben_szyld_2007}, our framework remains valid provided that the associated GLT symbol is invertible almost everywhere.}
The properties of GLT sequences that we need in this paper are listed below; for the corresponding proofs,
see \cite{Barbarino2020}.

\textbf{GLT0*.} Let $\{X_n\}_n$ be a sequence of matrices, with $X_n$ of size $d_n s \times d_n s$ for some fixed positive integer $s$ and some positive integer sequence $\{d_n\}_n$ tending to $\infty$, and let $\kappa, \xi : [0,1] \times [-\pi, \pi] \to \mathbb{C}^{s \times s}$ be measurable.
\begin{itemize}
	\item If $\{X_n\}_n \sim_\text{GLT} \kappa$ and $\kappa = \xi$ a.e., then $\{X_n\}_n \sim_\text{GLT} \xi$.
	\item If $\{X_n\}_n \sim_\text{GLT} \kappa$ and $\{X_n\}_n \sim_\text{GLT} \xi$, then $\kappa = \xi$ a.e.
\end{itemize}

\textbf{GLT1*.} Let $\{X_n\}_n$ be a sequence of matrices, with $X_n$ of size $d_n s \times d_n s$ for some fixed positive integers $s$ and some positive integer sequence $\{d_n\}_n$ tending to $\infty$, and let $\kappa : [0, 1] \times [-\pi, \pi] \to \mathbb{C}^{s \times s}$ be measurable.

\begin{itemize}
	\item If $\{X_n\}_n \sim_\text{GLT} \kappa$, then $\{X_n\}_n \sim_\sigma \kappa$.
	\item If $\{X_n\}_n \sim_\text{GLT} \kappa$ and the matrices $X_n$ are Hermitian, then $\kappa$ is Hermitian a.e., and $\{X_n\}_n \sim_\lambda \kappa$.
\end{itemize}

\textbf{GLT2*.} Let $s$ be a positive integer and let $\{d_n\}_n$ be a positive integer sequence tending to $\infty$. Then,

\begin{itemize}
	\item $\{T_{d_n}(f)\}_n \sim_\text{GLT} \kappa(x, \theta) = f(\theta)$ if $f : [-\pi, \pi] \to \mathbb{C}^{s \times s}$ belongs to $L^1([-\pi, \pi])$;
	\item $\{D_{d_n}(a)\}_n \sim_\text{GLT} \kappa(x, \theta) = a(x)$ if $a : [0, 1] \to \mathbb{C}^{s \times s}$ is continuous a.e.;
	\item For every sequence of matrices $\{Z_n\}_n$ with $Z_n$ of size $d_n s \times d_n s$, we have $\{Z_n\}_n \sim_\text{GLT} \kappa(x, \theta) = O_{s,t}$ if and only if $\{Z_n\}_n \sim_\sigma 0$.
\end{itemize}
\textbf{GLT3*.} Let $\{X_n\}_n$, $\{Y_n\}_n$ be sequences of matrices, with $X_n,Y_n$ of size $d_n s \times d_n s$ for some fixed positive integers $s$ and some positive integer sequence $\{d_n\}_n$ tending to $\infty$, and let $\kappa, \xi : [0,1] \times [-\pi, \pi] \to \mathbb{C}^{s \times s}$
be measurable. Suppose that $\{X_n\}_n \sim_\text{GLT} \kappa$ and $\{Y_n\}_n \sim_\text{GLT} \xi$. Then,

\begin{itemize}
	\item $\{X_n^*\}_n \sim_\text{GLT} \kappa^*$;
	\item $\{\alpha X_n + \beta Y_n\}_n \sim_\text{GLT} \alpha \kappa + \beta \xi$ for every $\alpha, \beta \in \mathbb{C}$;
	\item $\{X_n Y_n\}_n \sim_\text{GLT} \kappa \xi$;
	\item $\{X_n^\dagger\}_n \sim_\text{GLT} \kappa^{-1}$ if $\kappa$ is invertible a.e.;
	\item $\{f(X_n)\}_n \sim_\text{GLT} f(\kappa)$ if the matrices $X_n$ are Hermitian and $f : \mathbb{C} \to \mathbb{C}$ is continuous.
\end{itemize}

\textbf{GLT4*.} Let $\{X_n\}_n$ be a sequence of matrices, with $X_n$ of size $d_n s \times d_n s$ for some fixed positive integer $s$ and some positive integer sequence $\{d_n\}_n$ tending to $\infty$, and let $\kappa : [0,1] \times [-\pi, \pi] \to \mathbb{C}^{s \times s}$ be measurable. Suppose that there exists a sequence of sequences of matrices $\{\{X_{n,m}\}_n\}_m$ with $X_{n,m}$ of size $d_n s \times d_n s$ such that

\begin{itemize}
	\item $\{X_{n,m}\}_n \sim_\text{GLT} \kappa_m$ for some measurable $\kappa_m : [0,1] \times [-\pi, \pi] \to \mathbb{C}^{s \times s}$,
	\item $\kappa_m \to \kappa$ a.e. on $[0,1] \times [-\pi, \pi]$,
	\item $\{X_{n,m}\}_n \xrightarrow{\text{a.c.s.}} \{X_n\}_n$.
\end{itemize}

Then, $\{X_n\}_n \sim_\text{GLT} \kappa$.

We remind that the topological notion of \emph{approximating classes of sequences} (a.c.s.) is the cornerstone of an asymptotic approximation theory for sequences of matrices, that has been developed since the last 30 years and which represents the operative tool for building the GLT theory; see \cite{acs} for the first explicit definition, \cite[Chapter~5]{garoni2017generalized} for a general treatment, and \cite{gacs} for an extension and new results.

\section{One-Level GLT Sequences}\label{sec5:unified glt}

In the following, we restrict our attention to the one-level case of GLT sequences, corresponding to the setting $s = t \ge  1$ introduced earlier.
The statements and the mathematical derivations are done in the case of $s=t=1$. The general case of $s=t>1$ is treated in Remark \ref{s-t>1} and it is not academical since methods such as Discontinuous Galerkin, high order Finite Elements, Isogeometric Analysis with intermediate regularity, and hybrid methods lead to block structures and to matrix-valued GLT symbols \cite{GSS,doro,dumbser,tom}.

\begin{theorem}\label{non overlapping theorem}
	Let $\nu \geq 1$, be a fixed integer, and let $\{d_n\}_n$ be a sequence of positive integers tending to $\infty$. For every $n$, let
	\[
	(n_1, \ldots, n_{3\nu-2}) = (n_1(n), \ldots, n_{3\nu-2}(n))
	\]
	be a partition of $d_n$. Assume that $\left\{ A_n \right\}_n\sim_{\mathrm{GLT}} \kappa$ for some measurable function $\kappa$. Then
	\[
	\left\{ \operatorname{P_{\star, n}^{(\nu)}}(A_{n}) \right\}_n \sim_{\mathrm{GLT}} \kappa,
	\]
	where $\star \in \{\mathrm{BJ},\,\mathrm{BAS},\,  \mathrm{BGS},\, \mathrm{BMS}, \}$ and with $\nu$ non-overlapping subdomains.
	\end{theorem}
\begin{proof}
First of all we observe that the block Jacobi (BJ) preconditioning and the block Gauss Seidel (BGS) preconditioning coincide with the block additive
Schwarz (BAS) and with the block multiplicative Schwarz (BMS) preconditioning, of course also in their restricted version, since there is no overlapping among the $\nu$ subdomains; see Remark \ref{rem:no overlap ie BJ-BGS}.

The claimed results for the matrix-sequences $\{P_{\star,n}^{(\nu)}(A_n)\}_n$ and $\{P_n^{-1}A_n\}_n$, with $P_n = P_{\star,n}^{(\nu)}(A_n)$, are established in \cite[Theorem~4.1 and Theorem~4.2]{GLH} (see also the more general version for rectangular GLT symbols), while the corresponding results for the preconditioned matrix-sequences follow from Remark~4.1 in \cite{GLH}.

\end{proof}

The overlapping is definitely more involved and for that we need new ideas as already indicated in Remark \ref{rem:no overlap ie BJ-BGS}. For making them crystal clear, the proof of the subsequent three theorems contains all the details in the simplest case of two subdomains i.e. $\nu=2$. Theorem \ref{overlapping theorem - negative} contains a negative result concerning the block additive Schwarz preconditioning with overlapping, while its restricted version shows again the desired behavior. As anticipated in Remark \ref{rem:no need restriction in multiplicative}, the block multiplicative Schwarz does need to be corrected.

\begin{theorem}\label{overlapping theorem - negative}
	Let $\nu \geq 1$, be a fixed integer, and let $\{d_n\}_n$ be a sequence of positive integers tending to $\infty$. For every $n$, let
	\[
	(n_1, \ldots, n_{3\nu-2}) = (n_1(n), \ldots, n_{3\nu-2}(n))
	\]
	be a partition of $d_n$. Assume that $\left\{ A_n \right\}_n\sim_{\mathrm{GLT}} \kappa$ for some measurable function $\kappa$, invertible a.e. Then
	\[
	\left\{ \operatorname{P_{\star, n}^{(\nu)}}(A_{n}) \right\}_n \sim_{\mathrm{GLT}} \psi={\frac{\kappa}{\sum_{i=1}^\nu a_i}},
	\]
	where $\star = \mathrm{BAS}$, with $\nu$ subdomains with overlapping and with $a_i$ being the characteristic function of
	$\Omega_i$, $i=1,\ldots,\nu$.
\end{theorem}

\begin{proof}
Following (\ref{eq:add_mul_pre-nu}), we have
\[
T_{\text{BAS}, n}^{(\nu = 2)}(A_n) = I - \sum_{i=1}^{2} R_i^T A_i^{-1} R_i A_n,
\]
where
\[
R_1^T A_1^{-1} R_1=
	 	\begin{bmatrix}
		A_1^{-1} & O \\
		O & O
	\end{bmatrix}=\left[\mathrm{BlockDiag}_{\hat{d}_1(n),d_n-\hat{d}_1(n)}(A_n)\right]^{-1} \begin{bmatrix}
		I_{\hat{d}_1(n)} & O \\
		O & O
	\end{bmatrix},
\]
$\hat{d}_1(n)=d_1(n)+n_3$, and
\[
R_2^T A_2^{-1} R_2=\begin{bmatrix}
		O & O \\
		O & A_2^{-1}
	\end{bmatrix}=\left[\mathrm{BlockDiag}_{d_n-\hat{d}_2(n),\hat{d}_2(n)}(A_n)\right]^{-1} \begin{bmatrix}
		O & O \\
		O & I_{\hat{d}_2(n)}
	\end{bmatrix},
\]
$\hat{d}_2(n)=d_2(n)+n_2$, with $d_n=d_1(n)+d_2(n)$, $d_1(n)=n_1+n_2$, $d_2(n)=n_3+n_4$, so that $d_n-\hat{d}_1(n)=n_4$ and $d_n-\hat{d}_2(n)=n_1$.
The right-hand sides of the latter two equations is very informative from the GLT viewpoint since
\[
R_1^T A_1^{-1} R_1 =\left[\mathrm{BlockDiag}_{\hat{d}_1(n),d_n-\hat{d}_1(n)}(A_n)\right]^{-1} D_n(a_1),\quad \
R_2^T A_2^{-1} R_2 =\left[\mathrm{BlockDiag}_{d_n-\hat{d}_2(n),\hat{d}_2(n)}(A_n)\right]^{-1} D_n(a_2),
\]
$a_i$ being the characteristic function of the domain $\Omega_i$, $i=1,2$, $D_n(a_i)$, $i=1,2$, as in Axiom \textbf{GLT2*.}, second part,
for which
\[
\left\{ D_n(a_i) \right\}_n\sim_{\mathrm{GLT}} a_i,\quad i=1,2.
\]
Since $\left\{ A_n \right\}_n\sim_{\mathrm{GLT}} \kappa$, by Theorem \cite[{Theorem 4.1}]{GLH} we deduce
\[
\left\{\mathrm{BlockDiag}_{\hat{d}_1(n),d_n-\hat{d}_1(n)}(A_n) \right\}_n\sim_{\mathrm{GLT}} \kappa, \quad \
\left\{\mathrm{BlockDiag}_{d_n-\hat{d}_2(n),\hat{d}_2(n)}(A_n) \right\}_n\sim_{\mathrm{GLT}} \kappa.
\]
Therefore, using Axiom \textbf{GLT3*.}, fourth part, and Axiom \textbf{GLT3*.}, third part, we obtain
\[
\{R_1^T A_1^{-1} R_1\}_n\sim_{\mathrm{GLT}}  a_1\kappa^{-1},\quad\ \{R_2^T A_2^{-1} R_2\}_n\sim_{\mathrm{GLT}}  a_2\kappa^{-1},
\]
so that  $\left\{ \sum_{i=1}^{2} R_i^T A_i^{-1} R_i A_n\right\}_n\sim_{\mathrm{GLT}} a_1+a_2$, by invoking Axiom \textbf{GLT3*.}, second part and third part.
Hence
\[
\left\{ \operatorname{P_{\star, n}^{(\nu)}}(A_{n}) \right\}_n \sim_{\mathrm{GLT}} \psi={\frac{\kappa}{(a_1+a_2)}}.
\]
The case of $\nu>2$ subdomains follows the same steps, taking into consideration the admissibility conditions in Remark \ref{rem:non and overlap with nu subd}.
\end{proof}

\begin{theorem}\label{overlapping theorem - positive-M}
	Let $\nu \geq 1$, be a fixed integer, and let $\{d_n\}_n$ be a sequence of positive integers tending to $\infty$. For every $n$, let
	\[
	(n_1, \ldots, n_{3\nu-2}) = (n_1(n), \ldots, n_{3\nu-2}(n))
	\]
	be a partition of $d_n$. Assume that $\left\{ A_n \right\}_n\sim_{\mathrm{GLT}} \kappa$ for some measurable function $\kappa$, invertible a.e. Then
	\[
	\left\{ \operatorname{P_{\star, n}^{(\nu)}}(A_{n}) \right\}_n \sim_{\mathrm{GLT}} \kappa,
	\]
	where $\star =\mathrm{BMS}$, with $\nu$ subdomains with overlapping and with $a_i$ being the characteristic function of
	$\Omega_i$, $i=1,\ldots,\nu$.
\end{theorem}

\begin{proof}
For the multiplicative version we invoke the explicit expression in (\ref{mul_pre_only-nu}). In fact, the proof uses the same arguments as in Theorem \ref{overlapping theorem - negative}, using again the $*$-algebra structure of the GLT matrix-sequences.
Looking indeed at the related matrix-sequence $\left\{P_n^{-1}\right\}_n$ as in (\ref{eq:add_mul_pre-nu}), the GLT symbol is
\begin{eqnarray*}
\left(1-\left(1-a_2\kappa^{-1}\kappa\right)\left(1-a_1\kappa^{-1}\kappa\right) \right) \kappa^{-1} & = &
  \left(a_1+a_2 - a_1a_2\right) \kappa^{-1} = \kappa^{-1},
\end{eqnarray*}
since $a_1+a_2 - a_1a_2=1$ both in the overlapping and non-overlapping setting.

The case of $\nu>2$ subdomains follows the same steps, taking into consideration the admissibility conditions in Remark \ref{rem:non and overlap with nu subd}.
\end{proof}

\begin{theorem}\label{overlapping theorem - positive}
	Let $\nu \geq 1$, be a fixed integer, and let $\{d_n\}_n$ be a sequence of positive integers tending to $\infty$. For every $n$, let
	\[
	(n_1, \ldots, n_{3\nu-2}) = (n_1(n), \ldots, n_{3\nu-2}(n))
	\]
	be a partition of $d_n$. Assume that $\left\{ A_n \right\}_n\sim_{\mathrm{GLT}} \kappa$ for some measurable function $\kappa$. Then
	\[
	\left\{ \operatorname{P_{\star, n}^{(\nu)}}(A_{n}) \right\}_n \sim_{\mathrm{GLT}} \kappa,
	\]
	where $\star \in \{\mathrm{BRAS},\, \mathrm{BRMS} \}$ and with $\nu$ subdomains with overlapping.
\end{theorem}

\begin{proof}

	As for Theorem \ref{overlapping theorem - negative}, we start by considering the case of two subdomains, i.e., $\nu = 2$. The proof is contained in the following observation: the action of the restricted operators changes $a_1$ into $\hat a_1$ and $a_2$ into $\hat a_2$, $\hat a_i$, $i=1,2$, being the characteristic functions of ${\Omega_i\setminus S_i}$ , $i=1,2$, $S_1\cup S_2=\Omega_1\cap \Omega_2$, $S_1\cap S_2=\emptyset$. Therefore $\psi=\kappa\sum_{i=1}^\nu \hat a_i$ coincides with $\kappa$, the sum $\sum_{i=1}^\nu \hat a_i$ becoming again a partition of the unity as in the non-overlapping case.

	More in detail, let $\{A_{n}\}_n\sim_{\mathrm{GLT}} \kappa$ for some measurable function $\kappa$.
Following (\ref{eq:TRAS-nu}), we have
\[
T_{\text{BRAS}, n}^{(\nu = 2)}(A_n) = I - \sum_{i=1}^{2} \widehat{R}_i^T A_i^{-1} R_i A_n,
\]
where
\[
\widehat{R}_1^T A_1^{-1} R_1=
\begin{bmatrix}
		I_{{d}_1(n)} & O \\
		O & O
\end{bmatrix}
	 	\begin{bmatrix}
		A_1^{-1} & O \\
		O & O
	\end{bmatrix}=\left[\mathrm{BlockDiag}_{\hat{d}_1(n),d_n-\hat{d}_1(n)}(A_n)\right]^{-1} \begin{bmatrix}
		I_{\hat{d}_1(n)} & O \\
		O & O
	\end{bmatrix},
\]
$\hat{d}_1(n)=d_1(n)+n_3$, and
\[
\widehat{R}_2^T A_2^{-1} R_2=
\begin{bmatrix}
		O & O \\
		O & I_{{d}_2(n)}
\end{bmatrix}		
\begin{bmatrix}
		O & O \\
		O & A_2^{-1}
	\end{bmatrix}=\left[\mathrm{BlockDiag}_{d_n-\hat{d}_2(n),\hat{d}_2(n)}(A_n)\right]^{-1}
	\begin{bmatrix}
		O & O \\
		O & I_{\hat{d}_2(n)}
	\end{bmatrix},
\]
$\hat{d}_2(n)=d_2(n)+n_2$, with $d_n=d_1(n)+d_2(n)$, $d_1(n)=n_1+n_2$, $d_2(n)=n_3+n_4$, so that $d_n-\hat{d}_1(n)=n_4$ and $d_n-\hat{d}_2(n)=n_1$.
The right-hand sides of the latter two equations are very useful for deducing the GLT structure since
\begin{eqnarray*}
\widehat{R}_1^T A_1^{-1} R_1  & = & D_n(\hat a_1)\left[\mathrm{BlockDiag}_{\hat{d}_1(n),d_n-\hat{d}_1(n)}(A_n)\right]^{-1} D_n(a_1), \\
\widehat{R}_2^T A_2^{-1} R_2  & = & D_n(\hat a_2)\left[\mathrm{BlockDiag}_{d_n-\hat{d}_2(n),\hat{d}_2(n)}(A_n)\right]^{-1} D_n(a_2),
\end{eqnarray*}
$a_i$ being the characteristic function of the domain $\Omega_i$, $i=1,2$, $D_n(a_i)$, $D_n(\hat a_i)$, $i=1,2$, as in Axiom \textbf{GLT2*.}, second part, for which
\[
\left\{ D_n(a_i) \right\}_n\sim_{\mathrm{GLT}} a_i,\ \ \left\{ D_n(\hat a_i) \right\}_n\sim_{\mathrm{GLT}} \hat a_i,\quad i=1,2.
\]
Since $\left\{ A_n \right\}_n\sim_{\mathrm{GLT}} \kappa$, by Theorem {\cite[Theorem 4.1]{GLH}} we deduce
\[
\left\{\mathrm{BlockDiag}_{\hat{d}_1(n),d_n-\hat{d}_1(n)}(A_n) \right\}_n\sim_{\mathrm{GLT}} \kappa, \quad \
\left\{\mathrm{BlockDiag}_{d_n-\hat{d}_2(n),\hat{d}_2(n)}(A_n) \right\}_n\sim_{\mathrm{GLT}} \kappa.
\]
Therefore, using Axiom \textbf{GLT3*.}, fourth part, and Axiom \textbf{GLT3*.}, third part, we obtain
\[
\{\widehat{R}_1^T A_1^{-1} R_1\}_n\sim_{\mathrm{GLT}} \hat a_1 a_1\kappa^{-1}=\hat a_1 \kappa^{-1},
\quad\ \{\widehat{R}_2^T A_2^{-1} R_2\}_n\sim_{\mathrm{GLT}}
\hat a_2 a_2\kappa^{-1}=\hat a_2 \kappa^{-1},
\]
so that  $\left\{ \sum_{i=1}^{2} \widehat{R}_i^T A_i^{-1} R_i A_n\right\}_n\sim_{\mathrm{GLT}} \hat a_1+\hat a_2=1$, by invoking Axiom \textbf{GLT3*.}, second part and third part.
As a consequence, we have
\[
\left\{ \operatorname{P_{\star, n}^{(\nu)}}(A_{n}) \right\}_n \sim_{\mathrm{GLT}} \kappa.
\]

Due to the admissibility conditions in Remark \ref{rem:non and overlap with nu subd}, the case of $\nu>2$ subdomains can be treated in the same way since the overlapping holds at most for consecutive subdomains $\Omega_i$ and $\Omega_{i+1}$ with $i\in {\cal OV}$: hence the proof with $\nu>2$ is given by iterating the same reasoning as for $\nu=2$ for every $i\in {\cal OV}$.
\end{proof}

\begin{remark}\label{s-t>1}
	The statements of Theorems \ref{non overlapping theorem}, \ref{overlapping theorem - negative}, \ref{overlapping theorem - positive-M}, \ref{overlapping theorem - positive} have been restricted to scalar-valued GLT symbols, notwithstanding the fact that the underlying theory holds at the full level of generality \cite{Barbarino2020,Barbarino2020-bis,GSS,GLH}.
	Nevertheless, in a number of relevant settings -- Galerkin or collocation discretizations within Isogeometric Analysis with intermediate smoothness, hybrid schemes, Discontinuous Galerkin methods, or high-order Finite Elements -- the resulting matrix structures give rise to matrix-sequences that still fall under the GLT framework, but now with matrix-valued GLT symbols \cite{GSS,doro,dumbser,tom}.
	In these cases the dimension of the matrix-valued symbol equals $\mu (p-k)^d$ (as explained in the introductory part of \cite{gacs}, and see also \cite{Barbarino2020,Barbarino2020-bis,GSS,doro,dumbser,tom}), with $d$ denoting the dimension of the physical domain, $p$ the polynomial degree, $k$ the global regularity attained by the numerical solution, and $\mu$ the number of equations in the PDE system, the latter arising whenever the blocking strategy is employed for systems of PDEs \cite{blocking first,blocking num,JNUN - rational,JNUN - irrational,physical DD}. Beyond this, the setting can be further extended to graded meshes in general domains, as treated e.g. in \cite{graded meshes on nonrect,physical DD,SerraCapizzano2003,SerraCapizzano2006} and references therein.
	That said, revisiting Theorems \ref{non overlapping theorem}, \ref{overlapping theorem - negative}, \ref{overlapping theorem - positive-M}, \ref{overlapping theorem - positive}, one sees that extending them to the matrix-valued case raises no real difficulty: the characteristic functions entering the proofs are themselves scalar-valued, and therefore commute with any matrix-valued symbol, so the arguments carry over unchanged.
\end{remark}

\section{Numerical experiments}\label{sec6:num}

In this section, we demonstrate the efficiency of the block Jacobi/Gauss--Seidel and additive/multiplicative preconditioners for GLT sequences, in agreement with the predictions of our theory. Each example is presented following the procedure outlined below.

\begin{enumerate}

\item[$(A)$] We fix a GLT sequence \(\{A_n\}_n \sim_{\mathrm{GLT}} \kappa\), where each \(A_n\)
is an invertible square matrix of size \(n^d s \times n^d s\), for some fixed
positive integer \(s\), and
\(\kappa : [0,1]^d \times [-\pi,\pi]^d \to \mathbb{C}^{s \times s}\)
is a measurable matrix-valued function.

\item[$(B)$] For increasing values of $\nu$, $n$ and overalp size $o$, the tables report the number of iterations required to solve the linear system $A_n x = \mathbf{1}$ with a tolerance of $10^{-6}$. Results are provided for both the preconditioned conjugate gradient (PCG) method with $P_{\star,n}^{(\nu)}$ and the preconditioned GMRES (PGMRES) method. All solvers are initialized with the zero vector $\mathbf{0}$, and PGMRES is applied without restarting. Entries marked as "nac" indicate \textbf{cases that are not admissible}, this is because the overlap size $o$ cannot exceed the size of the subdomains, in accordance with the admissibility conditions stated in Remark~\eqref{rem:non and overlap with nu subd}. The results clearly confirm the efficiency of Schwarz-type preconditioners in accelerating convergence for both iterative methods, with convergence improving as the overlap increases.

\item [$(C)$] For $s>1$, we plot the graphs of selected continuous eigenvalue functions  $\lambda_i(f)$, $i = 1, \ldots,  s$, associated with the function $f$, together with the eigenvalues of the matrix $A_n$, for a selected value of $n$. The eigenvalues of $A_n$ are subdivided into
$s$ distinct subsets of approximately equal cardinality. The eigenvalues in each subset are then positioned at the nodes of a uniform grid in $[0,\pi]$.

\end{enumerate}

\subsection{The 1D setting}

\begin{example}[Full Toeplitz matrices]\label{F.T.M}
Consider the Toeplitz matrix \(A_n = T_n(f)\) generated by the symbol
\(f(\theta) = |\theta|\).
The corresponding matrix sequence \(\{A_n\}_n\) satisfies
$
\{A_n\}_n \sim_{\mathrm{GLT}} f,
$
by property \textbf{GLT2*}.
For item~(A), we fix the parameter \(s = 1\) and subsequently apply the
procedure described in item~(B).
The resulting numerical results, obtained via various methods, are
summarized in Tables~\ref{tab:1}--\ref{tab:14}.

Figure~\ref{fig:iterative_methods} provides a visual comparison of standard versus preconditioned GMRES. The left plot shows the residuals for unpreconditioned methods, while the right plot shows the residuals for preconditioned methods. The simulation is performed with matrix size $n = 1280$, number of subdomains $\nu = 2$, and overlap $o = 30$.

Figure~\ref{fig:eigenvalues_comparison} compares the eigenvalues for different block preconditioners with matrix size $n = 40$, number of subdomains $\nu = 2$, and overlap $o = 10$. The markers represent the eigenvalues of $A_n$, $P_n^{-1}$, and $P_n^{-1}A_n$ for each method, illustrating the clustering effect induced by the preconditioners and confirming the spectral distribution predictions based on GLT analysis.
\end{example}

\begin{table}[H]
\centering
\resizebox{0.8\textwidth}{!}{%
\begin{tabular}{ccccccccc}
\toprule
\textbf{Method} & $\nu$ & $n=40$ & $n=80$ & $n=160$ & $n=320$ & $n=640$ & $n=1280$ & $n=2560$ \\
\midrule
\multirow{4}{*}{PCG}
 & 2  & 5  & 6  & 6  & 6  & 7  & 7  & 7  \\
 & 4  & 8  & 9  & 10 & 10 & 11 & 12 & 12 \\
 & 8  & 12 & 13 & 14 & 16 & 17 & 18 & 19 \\
 &16  & 19 & 17 & 19 & 21 & 23 & 25 & 28 \\
\midrule
\multirow{4}{*}{PGMRES}
 & 2  & 5  & 6  & 6  & 9  & 7  & 10 & 11 \\
 & 4  & 8  & 12 & 10 & 15 & 16 & 19 & 18 \\
 & 8  & 13 & 14 & 17 & 19 & 21 & 22 & 30 \\
 &16  & 20 & 18 & 21 & 23 & 27 & 30 & 36 \\
\bottomrule
\end{tabular}
}
\caption{Number of iterations for PCG and PGMRES with $P_{BJ, n}^{(\nu )}(A_n)$.}
\label{tab:1}
\end{table}
\begin{table}[H]
\centering
\resizebox{0.8\textwidth}{!}{%
\begin{tabular}{ccccccccc}
\toprule
\textbf{Method} & $\nu$ & $n=40$ & $n=80$ & $n=160$ & $n=320$ & $n=640$ & $n=1280$ & $n=2560$ \\
\midrule

\multirow{4}{*}{PCG}
 & 2  & 4   & 5   & 5   & 5   & 6    & 6    & 6   \\
 & 4  & 8   & 8   & 9   & 9   & 9    & 10   & 11  \\
 & 8 & 8   & 10  & 11  & 12  & 13   & 14   & 15  \\
 & 16 & nac & 10  & 12  & 14  & 16   & 17   & 21  \\
\midrule

\multirow{4}{*}{PGMRES}
 & 2  & 4    & 6    & 5    & 7    & 9     & 6     & 10  \\
 &4  & 8    & 10   & 10   & 10   & 12    & 14    & 16  \\
 & 8  & 8    & 11   & 12   & 13   & 15    & 16    & 20  \\
 & 16 & nac   & 11   & 13   & 15   & 17    & 20    & 21  \\
\bottomrule
\end{tabular}
}
\caption{Number of iterations for PCG and PGMRES with $P_{BAS, n}^{(\nu )}(A_n)$, overlap $o=5$.}
\label{tab:2}
\end{table}

\begin{table}[H]
\centering
\resizebox{0.8\textwidth}{!}{%
\begin{tabular}{ccccccccc}
\toprule
\textbf{Method} & $\nu$ & $n=40$ & $n=80$ & $n=160$ & $n=320$ & $n=640$ & $n=1280$ & $n=2560$ \\
\midrule
\multirow{4}{*}{PCG}
 & 2  & 4   & 4   & 5   & 5   & 5    & 6    & 6   \\
 & 4  & 6   & 8   & 8   & 9   & 9    & 9    & 10  \\
 &8  & nac & 9   & 11  & 11  & 12   & 13   & 14  \\
 & 16 & nac & nac & 10  & 13  & 14   & 16   & 17  \\
\midrule

\multirow{4}{*}{PGMRES}
 & 2  & 5    & 4    & 6    & 5    & 5     & 9     & 6  \\
 & 4  & 7    & 8    & 11   & 10   & 10    & 12    & 14 \\
 & 8  & nac  & 9    & 13   & 16   & 12    & 15    & 17 \\
 & 16 & nac  & nac  & 11   & 14   & 15    & 17    & 20 \\
\bottomrule
\end{tabular} }
\caption{Number of iterations for PCG and PGMRES with $P_{BAS, n}^{(\nu )}(A_n)$ , with overlap $o=10$.}

\label{tab:3}

\end{table}


\begin{table}[H]
\centering
\resizebox{0.8\textwidth}{!}{%
\begin{tabular}{ccccccccc}
\toprule
\textbf{Method} & $\nu$ & $n=40$ & $n=80$ & $n=160$ & $n=320$ & $n=640$ & $n=1280$ & $n=2560$ \\
\midrule

\multirow{4}{*}{PCG}
 & 2 & nac & 3 & 4 & 4 & 5 & 5 & 6 \\
 & 4 & nac & nac & 8 & 8 & 8 & 9 & 9 \\
 & 8 & nac & nac & nac & 10 & 11 & 11 & 13 \\
 & 16 & nac & nac & nac & nac & 12 & 13 & 15 \\
\midrule

\multirow{4}{*}{PGMRES}
 & 2& nac & 3 & 4 & 4 & 5 & 5 & 7 \\
 & 4 & nac & nac & 8 & 10 & 8 & 9 & 13 \\
 & 8 & nac & nac & nac & 12 & 11 & 13 & 15 \\
 & 16 & nac & nac & nac & nac & 13 & 15 & 16 \\
\bottomrule
\end{tabular}}
 \caption{Number of iterations for PCG and PGMRES with $P_{BAS, n}^{(\nu )}(A_n)$, with overlap $o=30$.}
\label{tab:4}
\end{table}

\begin{table}[H]
\centering
\resizebox{0.8\textwidth}{!}{%
\begin{tabular}{llccccccc}
\toprule
\textbf{Method} & $\nu$ & $n=40$ & $n=80$ & $n=160$ & $n=320$ & $n=640$ & $n=1280$ & $n=2560$ \\

\midrule

\multirow{4}{*}{PGMRES} & 2 & 4 & 4 & 4 & 6 & 5 & 5 & 8 \\
 & 4 & 5 & 5 & 6 & 7 & 10 & 11 & 12 \\
 & 8 & 7 & 8 & 10 & 12 & 11 & 13 & 16 \\
 & 16 & nac & 11 & 11 & 12 & 15 & 17 & 20 \\
\bottomrule
\end{tabular}%
}
\caption{ Number of iterations for PGMRES with $P_{BRAS, n}^{(\nu )}(A_n)$, with overlap $o=5$.}
\label{tab:5}
\end{table}

\begin{table}[H]
\centering
\resizebox{0.8\textwidth}{!}{%
\begin{tabular}{llccccccc}
\toprule
\textbf{Method}  & $\nu$ & $n=40$ & $n=80$ & $n=160$ & $n=320$ & $n=640$ & $n=1280$ & $n=2560$ \\

\midrule
\multirow{4}{*}{PGMRES}
& 2 & 3 & 4 & 4 & 4 & 6 & 5 & 5 \\
 & 4 & 5 & 5 & 5 & 6 & 7 & 11 & 12 \\
 & 8 & nac & 8 & 7 & 8 & 13 & 12 & 14 \\
 & 16 & nac & nac & 12 & 11 & 12 & 15 & 18 \\
\bottomrule
\end{tabular}%
}
\caption{Number of iterations for PCG and PGMRES with $P_{BRAS, n}^{(\nu )}(A_n)$, with overlap $o=10$.}
\label{tab:6}
\end{table}

\begin{table}[H]
\centering
\resizebox{0.8\textwidth}{!}{%
\begin{tabular}{llccccccc}
\toprule
\textbf{Method} & $\nu$ & $n=40$ & $n=80$ & $n=160$ & $n=320$ & $n=640$ & $n=1280$ & $n=2560$ \\
\midrule

\multirow{4}{*}{PGMRES}
 & 2  & nac & 3   & 3   & 4   & 4   & 4    & 7 \\
 & 4  & nac & nac & 5   & 5   & 8   & 8    & 9 \\
 & 8  & nac & nac & nac & 8   & 8   & 9    & 11 \\
 & 16 & nac & nac & nac & nac & 11  & 12   & 14 \\
\bottomrule
\end{tabular}%
}
 \caption{Number of iterations for PCG and PGMRES with $P_{BRAS, n}^{(\nu )}(A_n)$, with overlap $o=30$.}
\label{tab:7}
\end{table}

\begin{table}[!htbp]
\centering
\resizebox{0.8\textwidth}{!}{%
\begin{tabular}{ccccccccc}
\toprule
Method & $\nu$ & $n=40$ & $n=80$ & $n=160$ & $n=320$ & $n=640$ & $n=1280$ & $n=2560$ \\
\midrule
\multirow{4}{*}{PGMRES}
 & 2 & 4 & 5 & 7 & 5 & 5 & 8 & 9 \\
 & 4 & 7 & 10 & 8 & 11 & 9 & 14 & 17 \\
 & 8 & 10 & 13 & 14 & 16 & 16 & 19 & 17 \\
 & 16 & 15 & 15 & 17 & 19 & 21 & 22 & 25 \\
\bottomrule
\end{tabular}%
}
\caption{Number of iterations for PGMRES with $P_{BGS, n}^{(\nu )}(A_n)$.}
\label{tab:8}
\end{table}

\begin{table}[H]
\centering
\resizebox{0.8\textwidth}{!}{%
\begin{tabular}{llccccccc}
\toprule
\textbf{Method} & $\nu$ & $n=40$ & $n=80$ & $n=160$ & $n=320$ & $n=640$ & $n=1280$ & $n=2560$ \\
\midrule

\multirow{4}{*}{PGMRES}& 2 & 3 & 3 & 4 & 4 & 4 & 4 & 6 \\
 & 4 & 4 & 5 & 7 & 7 & 7 & 8 & 8 \\
 & 8 & 5 & 7 & 8 & 9 & 11 & 12 & 14 \\
 & 16 & nac & 8 & 11 & 14 & 15 & 15 & 17 \\
\bottomrule
\end{tabular}%
}
\caption{Number of iterations for PGMRES with $P_{BMS, n}^{(\nu )}(A_n)$, with overlap $o=5$.}
\label{tab:9}
\end{table}

\begin{table}[H]
\centering
\resizebox{0.8\textwidth}{!}{%
\begin{tabular}{llccccccc}
\toprule
\textbf{Method} & $\nu$ & $n=40$ & $n=80$ & $n=160$ & $n=320$ & $n=640$ & $n=1280$ & $n=2560$ \\
\midrule

\multirow{4}{*}{PGMRES}& 2 & 3 & 3 & 3 & 4 & 4 & 4 & 5 \\
 & 4 & 3 & 4 & 5 & 7 & 7 & 7 & 9 \\
 & 8 & nac & 5 & 7 & 9 & 9 & 11 & 12 \\
 & 16 & nac & nac & 8 & 11 & 14 & 15 & 15 \\
\bottomrule
\end{tabular}}
 \caption{Number of iterations for GMRES with $P_{BMS, n}^{(\nu )}(A_n)$, with overlap $o=10$.}
\label{tab:10}
\end{table}

\begin{table}[H]
\centering
\resizebox{0.8\textwidth}{!}{%
\begin{tabular}{llccccccc}
\toprule
\textbf{Method} & $\nu$ & $n=40$ & $n=80$ & $n=160$ & $n=320$ & $n=640$ & $n=1280$ & $n=2560$ \\
\midrule
\multirow{4}{*}{PGMRES} & 2 & nac & 2 & 3 & 3 & 3 & 4 & 4 \\
 & 4 & nac & nac & 4 & 5 & 5 & 6 & 9 \\
 & 8 & nac & nac & nac & 6 & 8 & 9 & 9 \\
 & 16 & nac & nac & nac & nac & 9 & 13 & 15 \\
\bottomrule
\end{tabular}}
\caption{Number of iterations for PGMRES with $P_{BMS, n}^{(\nu )}(A_n)$, with overlap $o=30$.}
\label{tab:11}
\end{table}

\begin{table}[H]
\centering
\resizebox{0.8\textwidth}{!}{%
\begin{tabular}{llccccccc}
\toprule
\textbf{Method} & $\nu$ & $n=40$ & $n=80$ & $n=160$ & $n=320$ & $n=640$ & $n=1280$ & $n=2560$ \\

\midrule

\multirow{4}{*}{PGMRES} & 2 & 3 & 3 & 5 & 4 & 4 & 5 & 6 \\
 & 4 & 6 & 5 & 6 & 7 & 7 & 8 & 8 \\
 & 8 & 7 & 8 & 10 & 9 & 11 & 13 & 15 \\
 & 16 & nac & 10 & 14 & 15 & 16 & 16 & 17 \\
\bottomrule
\end{tabular}}
 \caption{Number of iterations for PGMRES with $P_{BRMS, n}^{(\nu )}(A_n)$, with overlap $o=5$.}
\label{tab:12}
\end{table}

\begin{table}[H]
\centering
\resizebox{0.8\textwidth}{!}{%
\begin{tabular}{llccccccc}
\toprule
\textbf{Method} & $\nu$ & $n=40$ & $n=80$ & $n=160$ & $n=320$ & $n=640$ & $n=1280$ & $n=2560$ \\
\midrule

\multirow{4}{*}{PGMRES} & 2 & 3 & 3 & 4 & 5 & 4 & 4 & 5 \\
 & 4 & 4 & 5 & 5 & 6 & 8 & 7 & 9 \\
 & 8 & nac & 7 & 8 & 10 & 9 & 12 & 13 \\
 & 16 & nac & nac & 10 & 14 & 15 & 16 & 16 \\
\bottomrule
\end{tabular} }
 \caption{Number of iterations for PGMRES with $P_{BRMS, n}^{(\nu )}(A_n)$, with overlap $o=10$.}
\label{tab:13}
\end{table}

\begin{table}[H]
\centering
\resizebox{0.8\textwidth}{!}{%
\begin{tabular}{llccccccc}
\toprule
\textbf{Method} & $\nu$ & $n=40$ & $n=80$ & $n=160$ & $n=320$ & $n=640$ & $n=1280$ & $n=2560$ \\
\midrule

\multirow{4}{*}{PGMRES}  & 2 & nac & 3 & 3 & 3 & 5 & 4 & 4 \\
 & 4 & nac & nac & 5 & 5 & 7 & 6 & 7 \\
 & 8 & nac & nac & nac & 8 & 9 & 9 & 11 \\
 & 16 & nac & nac & nac & nac & 12 & 15 & 16 \\
\bottomrule
\end{tabular}}
 \caption{Number of iterations for PGMRES with $P_{BRMS, n}^{(\nu )}(A_n)$, with overlap $o=30$.}
\label{tab:14}
\end{table}

\begin{figure}[H]
    \centering
    \includegraphics[width=.6\textwidth]{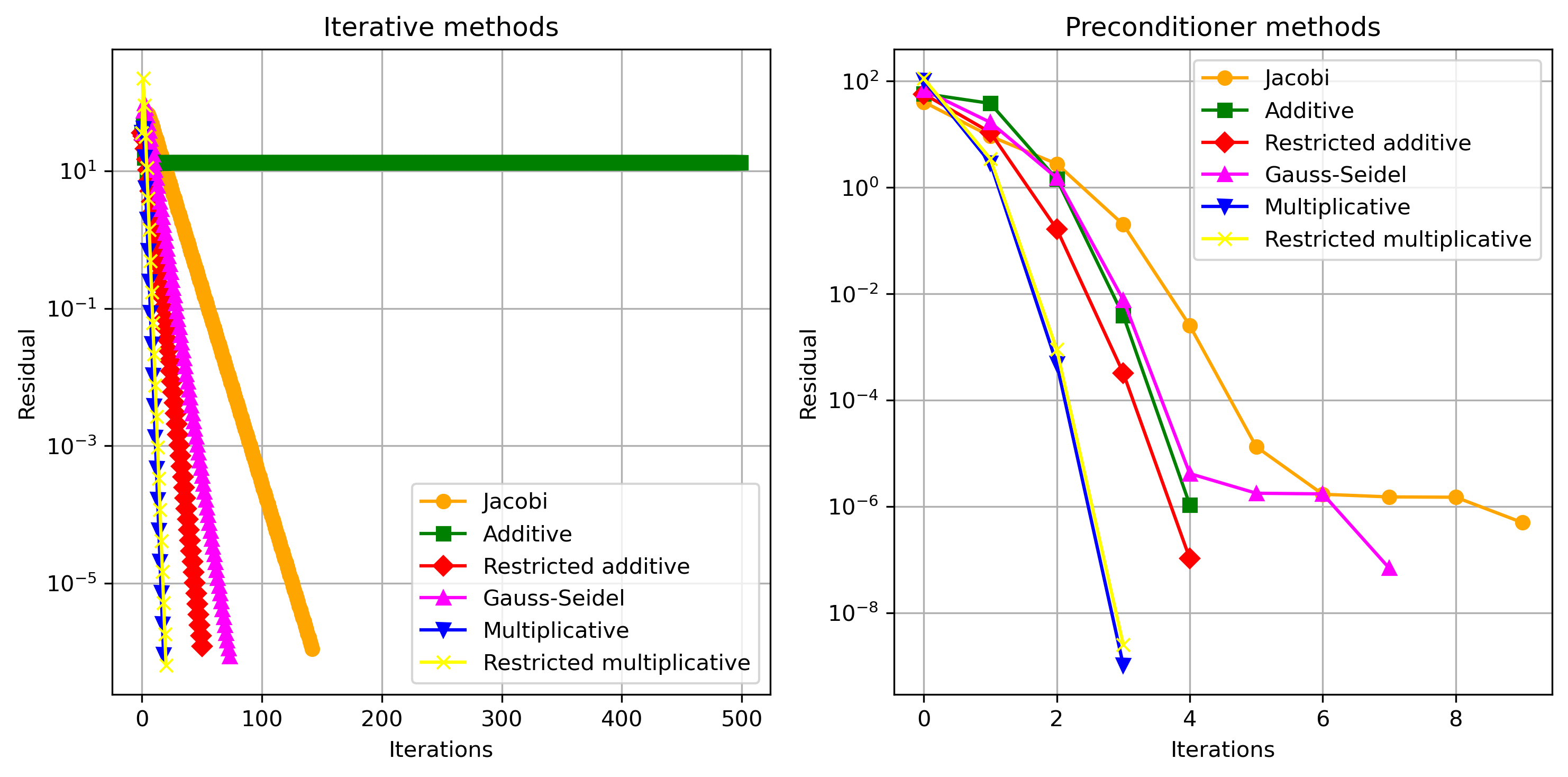}
    \caption{Comparison of iterative and preconditioner methods using GMRES. The left plot shows the residuals for standard iterative methods, while the right plot shows residuals for preconditioned methods. The simulation is performed with matrix size $n=1280$, number of subdomains $\nu=2$, and overlap $o=30$.}
    \label{fig:iterative_methods}
\end{figure}

\begin{figure}[H]
    \centering
    \includegraphics[width=0.6\textwidth]{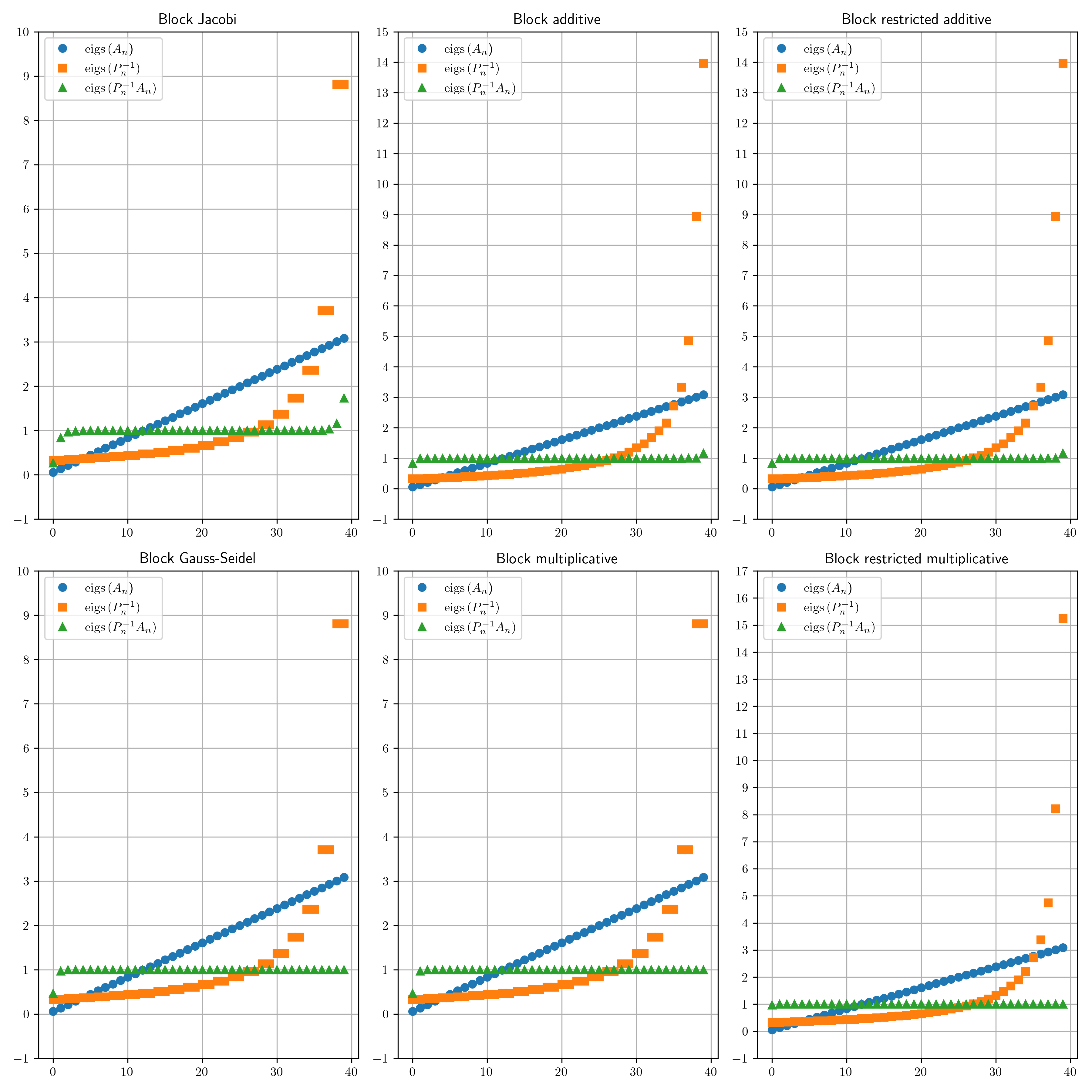}
    \caption{Comparison of eigenvalues for different block methods with matrix size $n=40$, number of subdomains $\nu=2$, and overlap $o=10$. The markers represent the eigenvalues of $A_n$, $P_n^{-1}$, and $P_n^{-1}A_n$ for each method.}
    \label{fig:eigenvalues_comparison}
\end{figure}

\begin{remark}
Both multiplicative and restricted additive Schwarz preconditioners are, in general, not Hermitian. This point is crucial when selecting an iterative solver: GMRES can be applied directly, whereas CG requires an HPD preconditioner.
\end{remark}

\begin{example}[Finite element methods with $s = 1$]\label{FEM}
Consider the one-dimensional elliptic boundary value problem
\begin{equation}\label{eq:poisson1}
\begin{cases}
-\bigl(a(x) u'(x)\bigr)' = f(x), & x \in (0,1),\\[2mm]
u(0) = u(1) = 0.
\end{cases}
\end{equation}
We assume that the diffusion coefficient \(a : [0,1] \to \mathbb{R}\) is
continuous and uniformly positive, that is,
\(
a \in C([0,1]),\,
 a(x)>0
\, \text{for all } x \in [0,1],
\)
moreover, we assume that the source term satisfies \(f \in L^2(0,1)\).

The weak formulation of \eqref{eq:poisson1} reads as follows: find
\(u \in H_0^1(0,1)\) such that
\begin{equation}
\int_0^1 a(x)\,u'(x)\,v'(x)\,dx
=
\int_0^1 f(x)\,v(x)\,dx,
\qquad \forall\, v \in H_0^1(0,1).
\end{equation}

We now introduce a standard finite element discretization. Let $n \in \mathbb{N}$ and set
\[
h = \frac{1}{n+1}, \qquad x_i = i h, \quad i = 0,\ldots,n+1,
\]
which defines a uniform partition of the interval $[0,1]$,
\(
0 = x_0 < x_1 < \cdots < x_{n+1} = 1.
\)
We consider the classical linear finite element space
\[
V_h
=
\left\{
v \in H_0^1(0,1)
\;\middle|\;
v|_{[x_i,x_{i+1})} \in \mathbb{P}_1,
\quad i = 0,\ldots,n
\right\},
\]
which consists of continuous, piecewise affine functions vanishing at the boundary. The space $V_h$ admits the nodal basis
\(
V_h = \operatorname{span}\{\varphi_1,\ldots,\varphi_n\},
\)
where $\varphi_i$ is the standard hat function associated with the interior node $x_i$, defined by
\begin{equation}\label{eq:hat-functions}
\varphi_i(x) =
\begin{cases}
\dfrac{x - x_{i-1}}{x_i - x_{i-1}}, & x \in [x_{i-1}, x_i), \\
\dfrac{x_{i+1} - x}{x_{i+1} - x_i}, & x \in [x_i, x_{i+1}), \\
0, & \text{otherwise},
\end{cases}
\qquad i = 1,\ldots,n.
\end{equation}
The finite element approximation of \eqref{eq:poisson1} consists in finding $u_h \in V_h$ such that
\begin{equation}
\int_0^1 a(x)\,u_h'(x)\,v'(x)\,dx
=
\int_0^1 f(x)\,v(x)\,dx,
\qquad \forall\, v \in V_h.
\end{equation}
Since $\{\varphi_1,\ldots,\varphi_n\}$ forms a basis of $V_h$, the discrete solution can be written as
\(
u_h(x) = \sum_{j=1}^n u_j \,\varphi_j(x),
\)
where $u = (u_1,\ldots,u_n)^T \in \mathbb{R}^n$ is the vector of unknown coefficients.
By linearity, the finite element problem reduces to solving the linear system
\begin{equation}
K_nu = f,
\end{equation}
where the load vector $f \in \mathbb{R}^n$ is given by
\(
f_i = \int_0^1 f(x)\,\varphi_i(x)\,dx,
\; i = 1,\ldots,n,
\)
and the stiffness matrix $K_n \in \mathbb{R}^{n \times n}$ has entries
\(
(K_n){ij} = \int_0^1 a(x)\,\varphi_j'(x)\,\varphi_i'(x)\,dx,
\; i,j = 1,\ldots,n.
\)

Here, for each \( n \in \mathbb{N} \), the following is also known in this case:

\begin{equation}\label{GLT:relations3}
    \left\{\frac{1}{n+1}K_n\right\}_n \sim_{\mathrm{GLT}} a(x)\,\bigl(2 - 2\cos\theta\bigr).
\end{equation}
\begin{equation}\label{GLT:relation4}
    \left\{\frac{1}{n+1}K_n\right\}_n \sim_{\mathrm{\lambda,\sigma}} a(x)\,\bigl(2 - 2\cos\theta\bigr).
\end{equation}
Therefore, both the case-by-case theoretical results established for each configuration and the numerical evidence reported in the accompanying figures confirm that all classical block preconditioners exhibit the predicted asymptotic behavior
\[
\{P_{\star,n}^{(\nu)}\}_n \sim_{\mathrm{GLT}} a(x)\,\bigl(2 - 2\cos\theta\bigr),
\qquad
\star \in \{\mathrm{BJ}, \mathrm{BAS}, \mathrm{BRAS}, \mathrm{BGS},
\mathrm{BMS}, \mathrm{BRMS}\}.
\]
This conclusion is in full agreement with the numerical observations
reported in Figures~\ref{Fig: FE 2}--\ref{Fig: FE 11}.

\begin{figure}[H]
    \centering
    \includegraphics[width=0.49\textwidth]{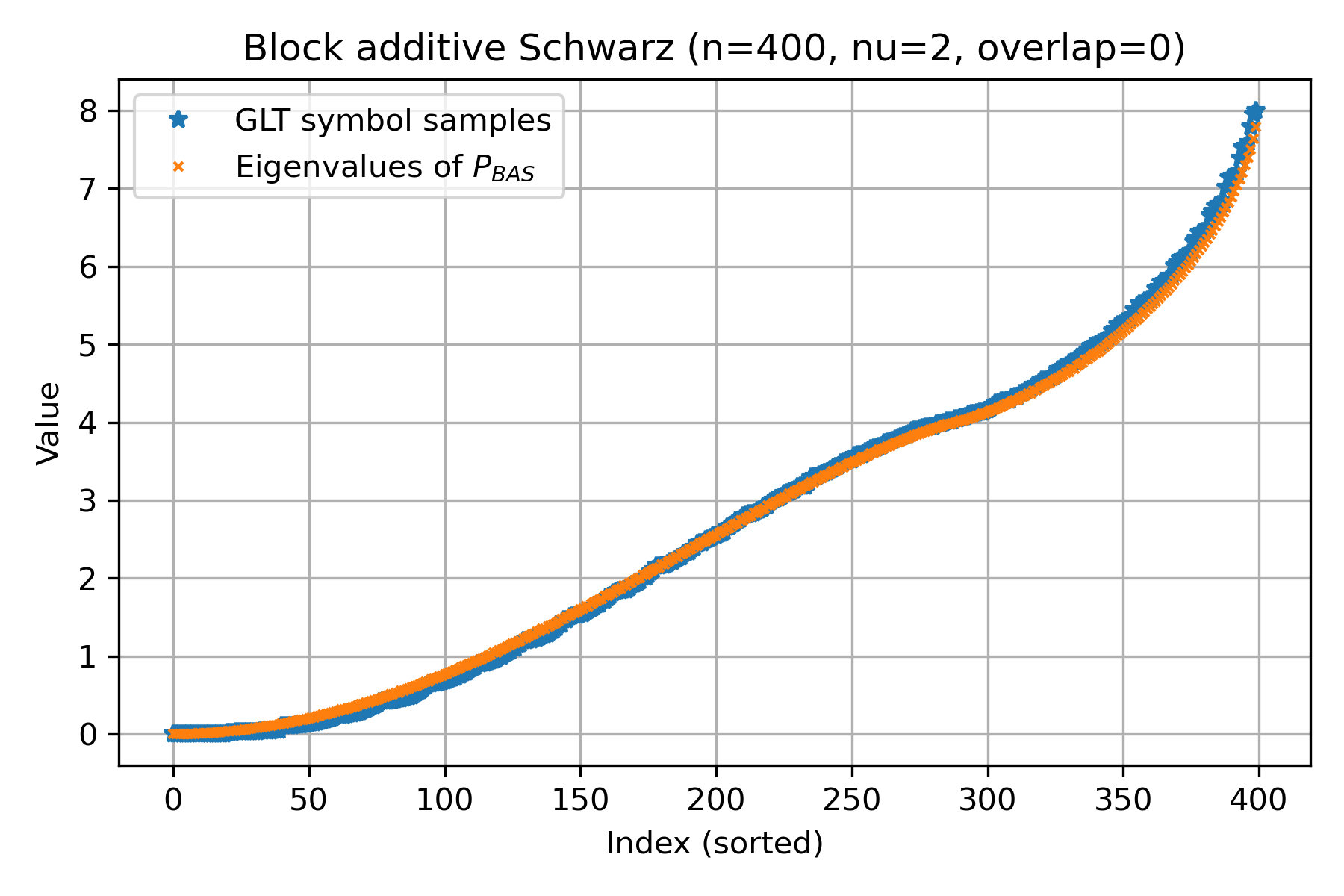}
    \includegraphics[width=0.49\textwidth]{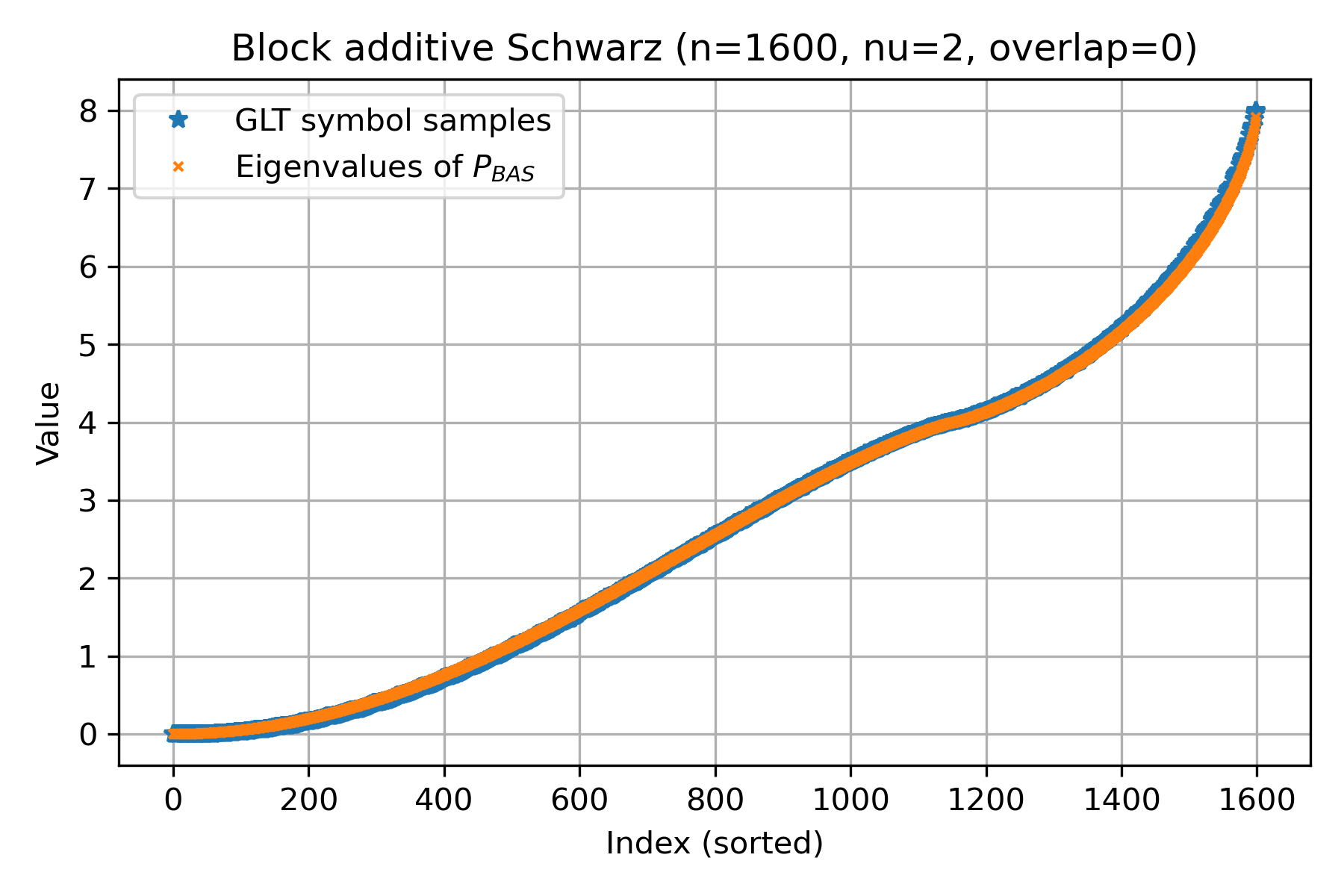}
    \caption{Comparison between the GLT symbol $(1+x^2)\,(2 - 2 \cos \theta)$ and the eigenvalues of the block additive Schwarz preconditioner \(P_{\mathrm{BAS},n}^{(\nu)}\) without overlap.
Left: $n=400$, $\nu=2$. Right: $n=1600$, $\nu=2$.
}

    \label{Fig: FE 2}
\end{figure}

\begin{figure}[H]
    \centering
    \includegraphics[width=0.49\textwidth]{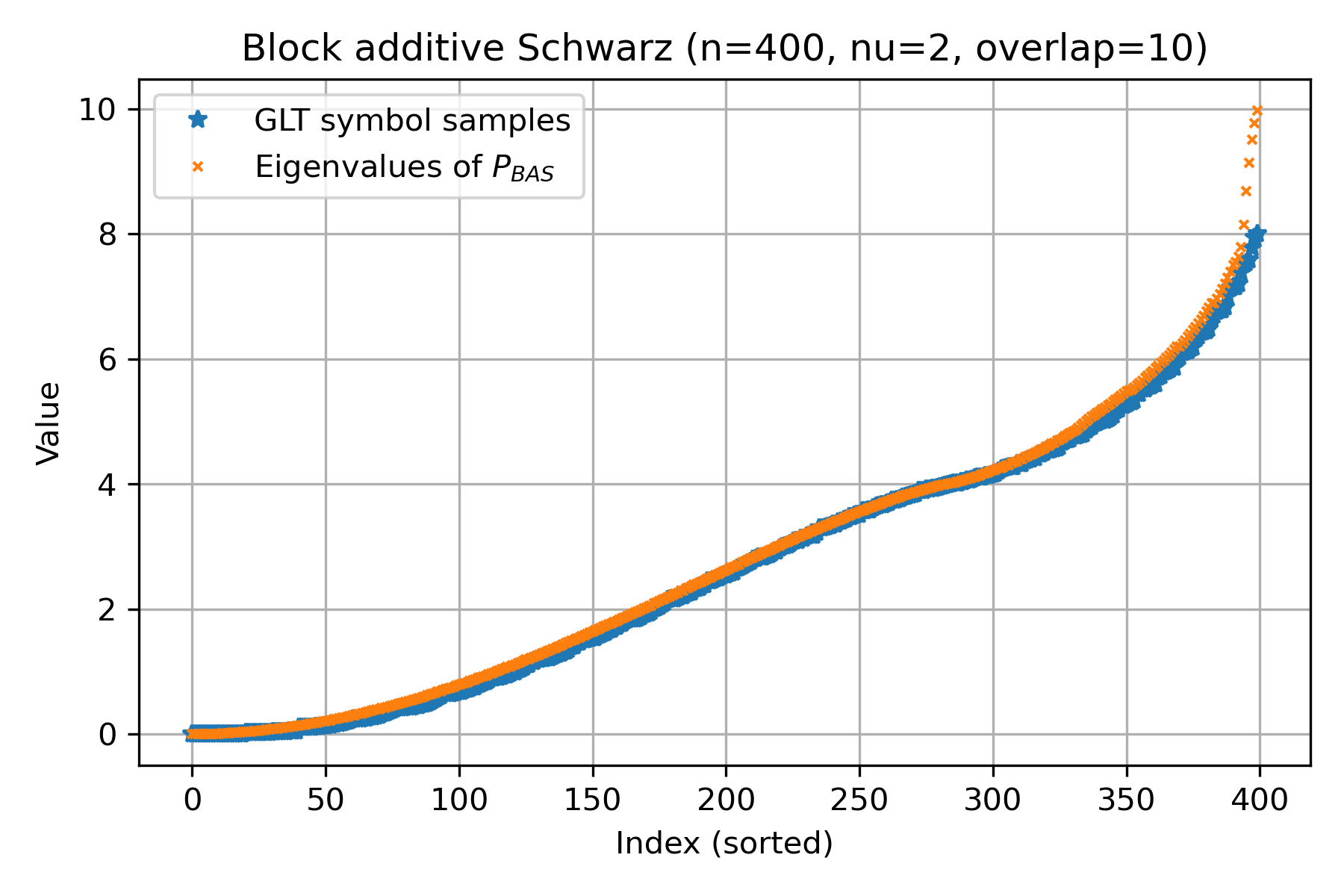}
    \includegraphics[width=0.49\textwidth]{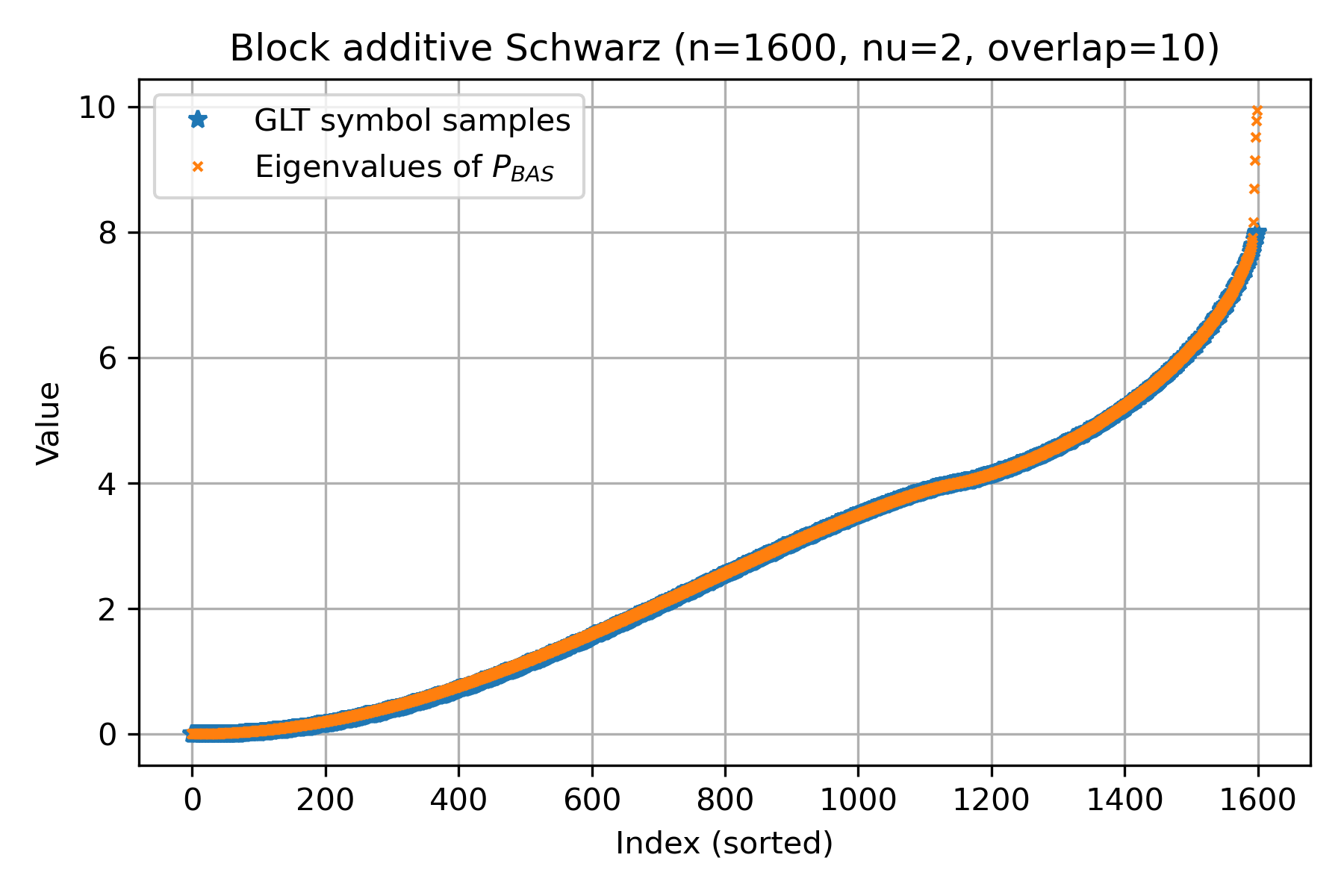}
    \caption{Comparison between the GLT symbol \((1+x^2)\,(2 - 2 \cos \theta)\) and the eigenvalues of the block additive Schwarz preconditioner \(P_{\mathrm{BAS},n}^{(\nu)}\) with overlap $o=10$. Left:$n=400$, $\nu=2$. Right: $n=1600$, $\nu=2$.}

    \label{Fig: FE 4}
\end{figure}

\begin{figure}[H]
    \centering
    \includegraphics[width=0.49\textwidth]{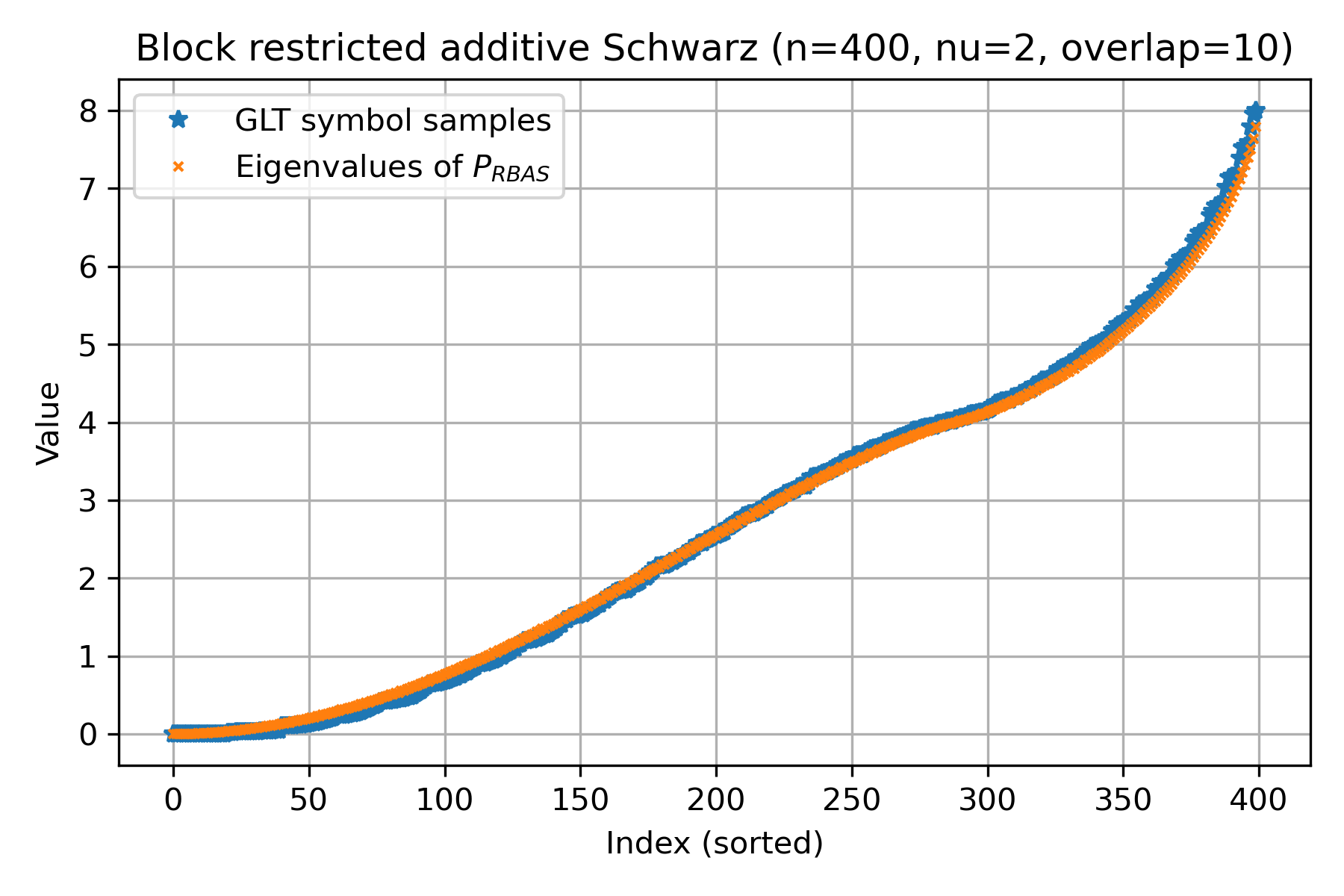}
    \includegraphics[width=0.49\textwidth]{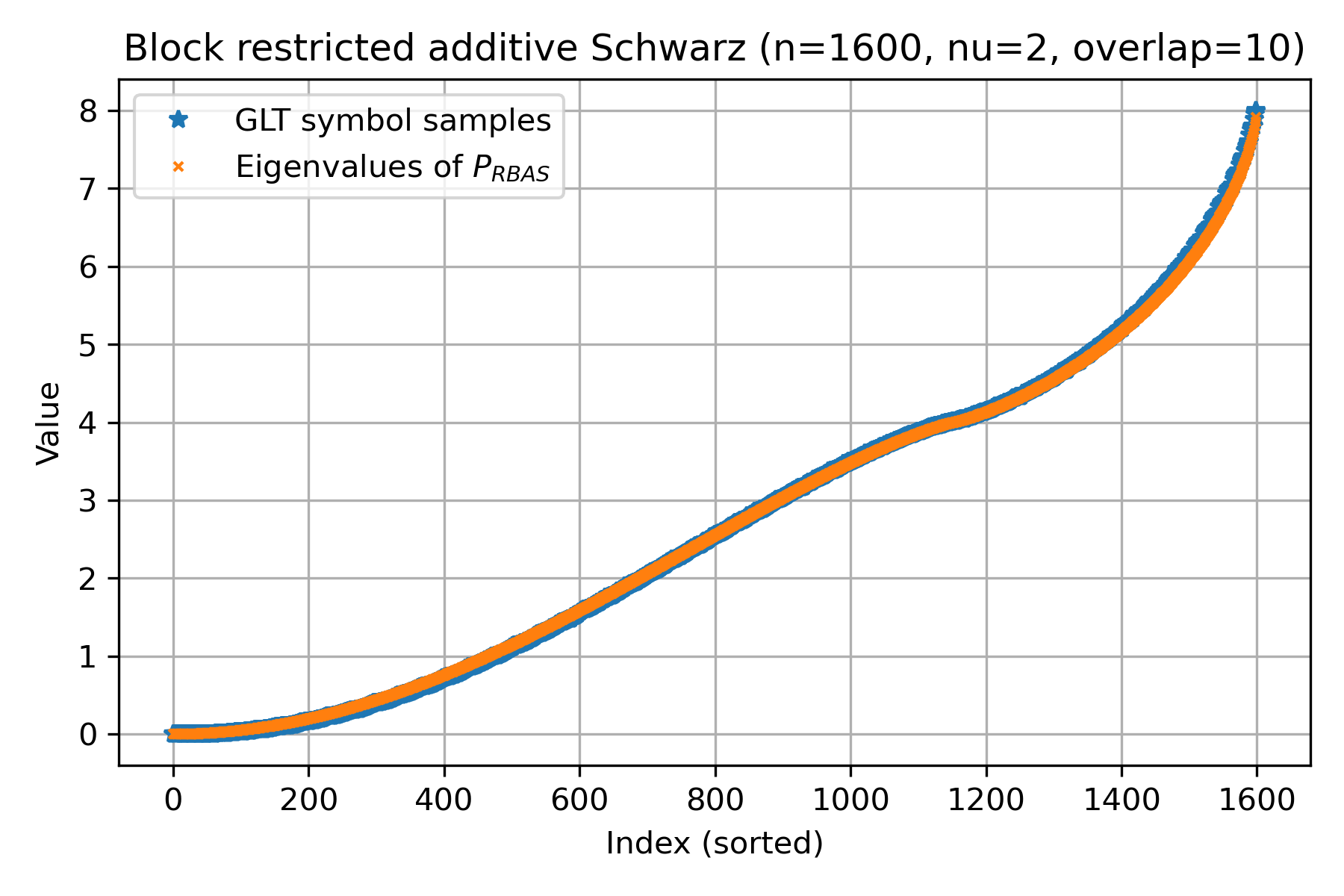}
    \caption{Comparison between the GLT symbol \((1+x^2)\,(2 - 2 \cos \theta)\) and the eigenvalues of the block restricted additive Schwarz preconditioner \(P_{\mathrm{BRAS},n}^{(\nu)}\) with overlap $o=10$.
Left: $n=400$, $\nu=2$. Right: $n=1600$, $\nu=2$.}

    \label{Fig: FE 6}
\end{figure}

\begin{figure}[H]
    \centering
    \includegraphics[width=0.49\textwidth]{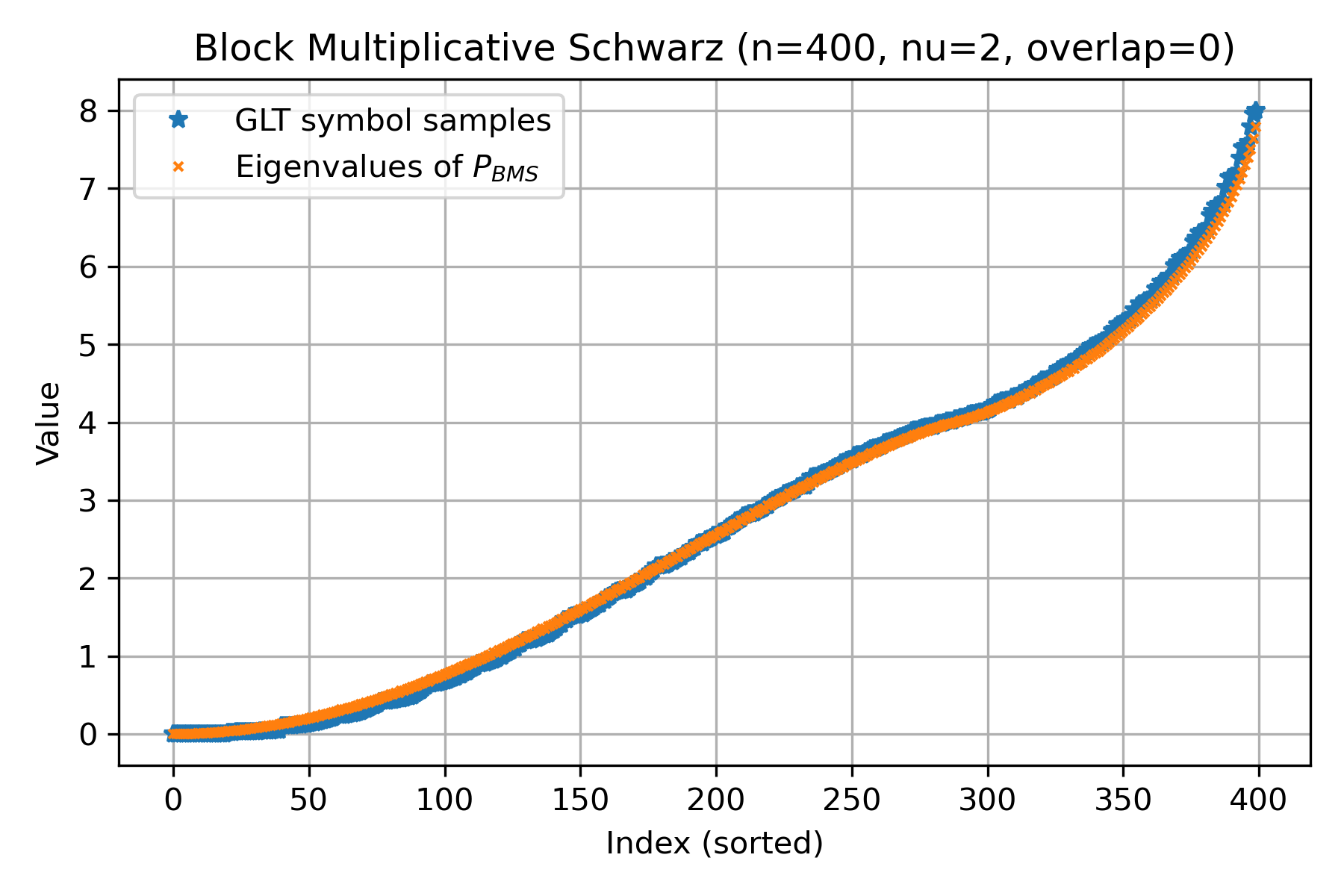}
    \includegraphics[width=0.49\textwidth]{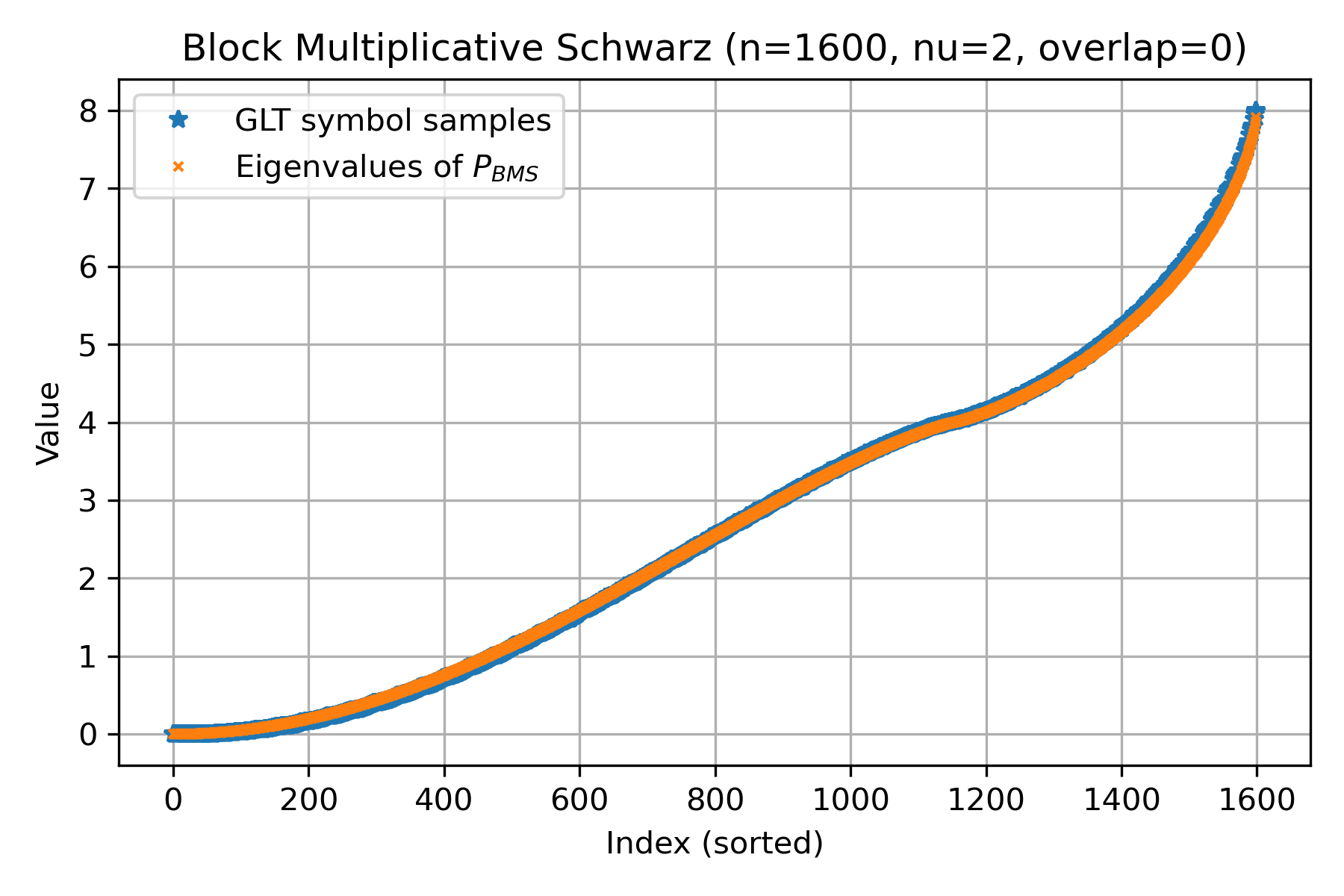}
    \caption{Comparison between the GLT symbol \((1+x^2)\,(2 - 2 \cos \theta)\) and the eigenvalues of the block multiplicative Schwarz preconditioner \(P_{\mathrm{BMS},n}^{(\nu)}\) without overlap.
Left: $n=400$, $\nu=2$. Right: $n=1600$, $\nu=2$.}

    \label{Fig7}
\end{figure}



\begin{figure}[H]
    \centering
    \includegraphics[width=0.49\textwidth]{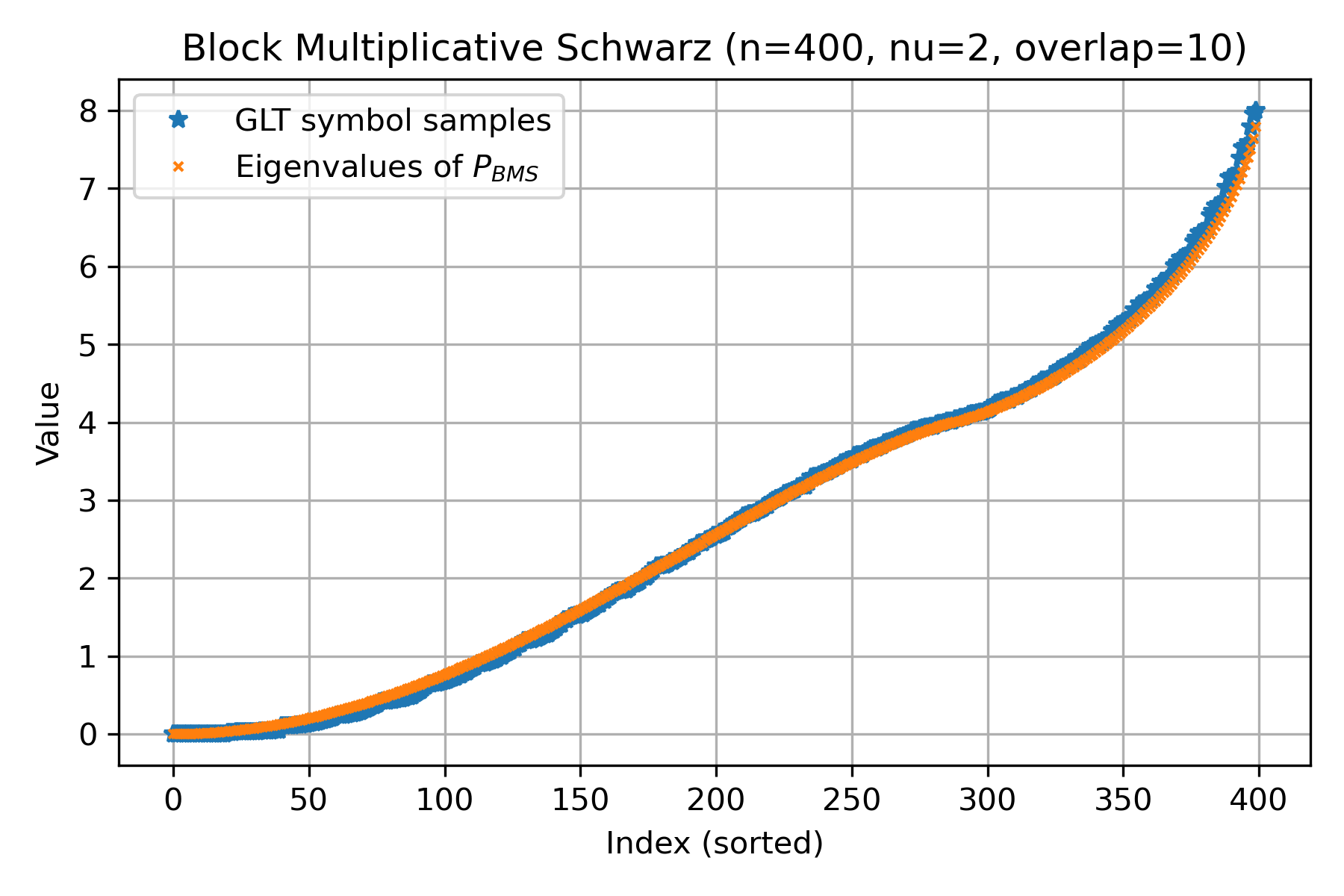}
    \includegraphics[width=0.49\textwidth]{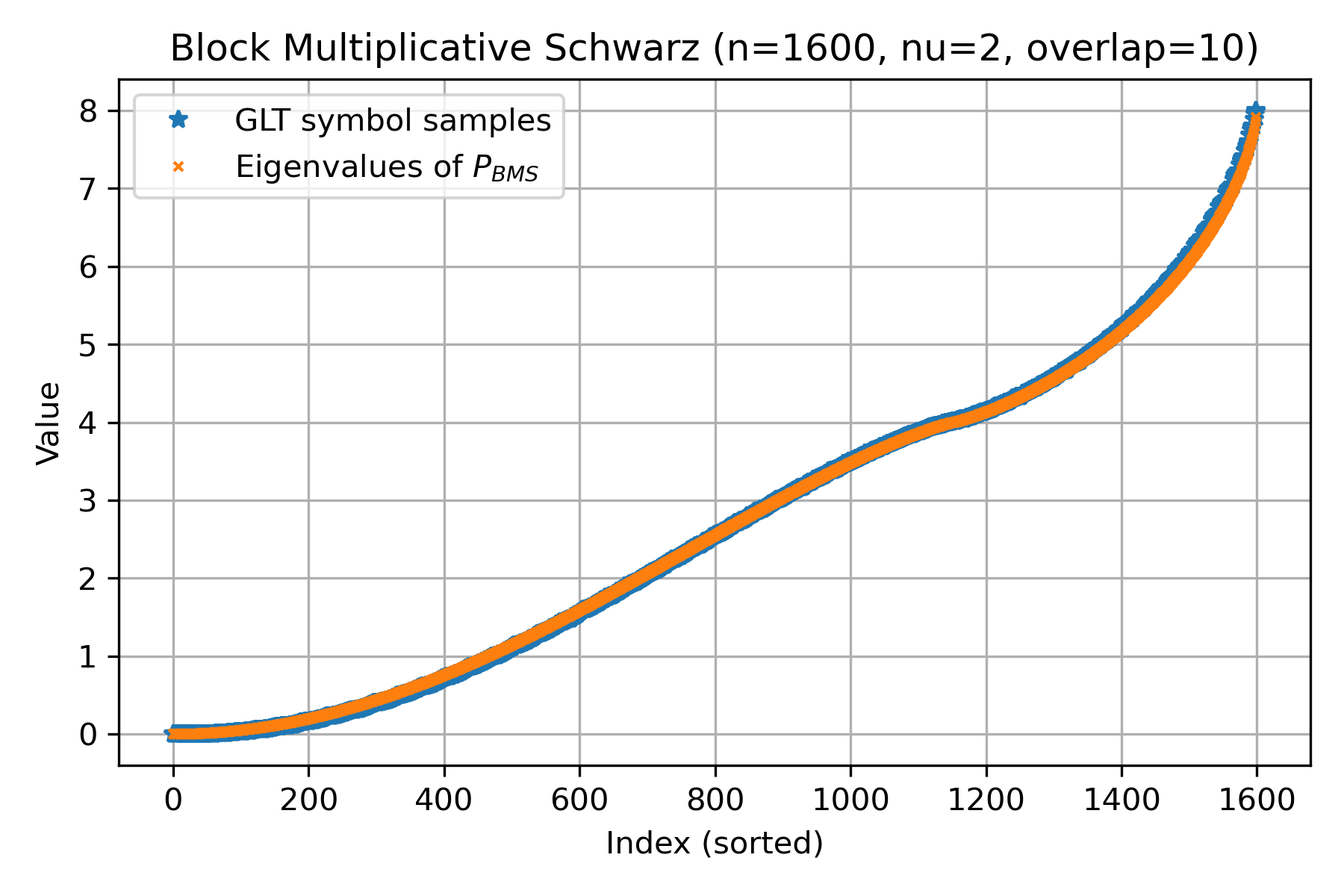}
    \caption{Comparison between the GLT symbol \((1+x^2)\,(2 - 2 \cos \theta)\) and the eigenvalues of the block multiplicative Schwarz preconditioner \(P_{\mathrm{BMS},n}^{(\nu)}\) with overlap $o=10$.
Left: $n=400$, $\nu=2$. Right: $n=1600$, $\nu=2$.}

    \label{Fig: FE 9}
\end{figure}

\begin{figure}[H]
    \centering
    \includegraphics[width=0.49\textwidth]{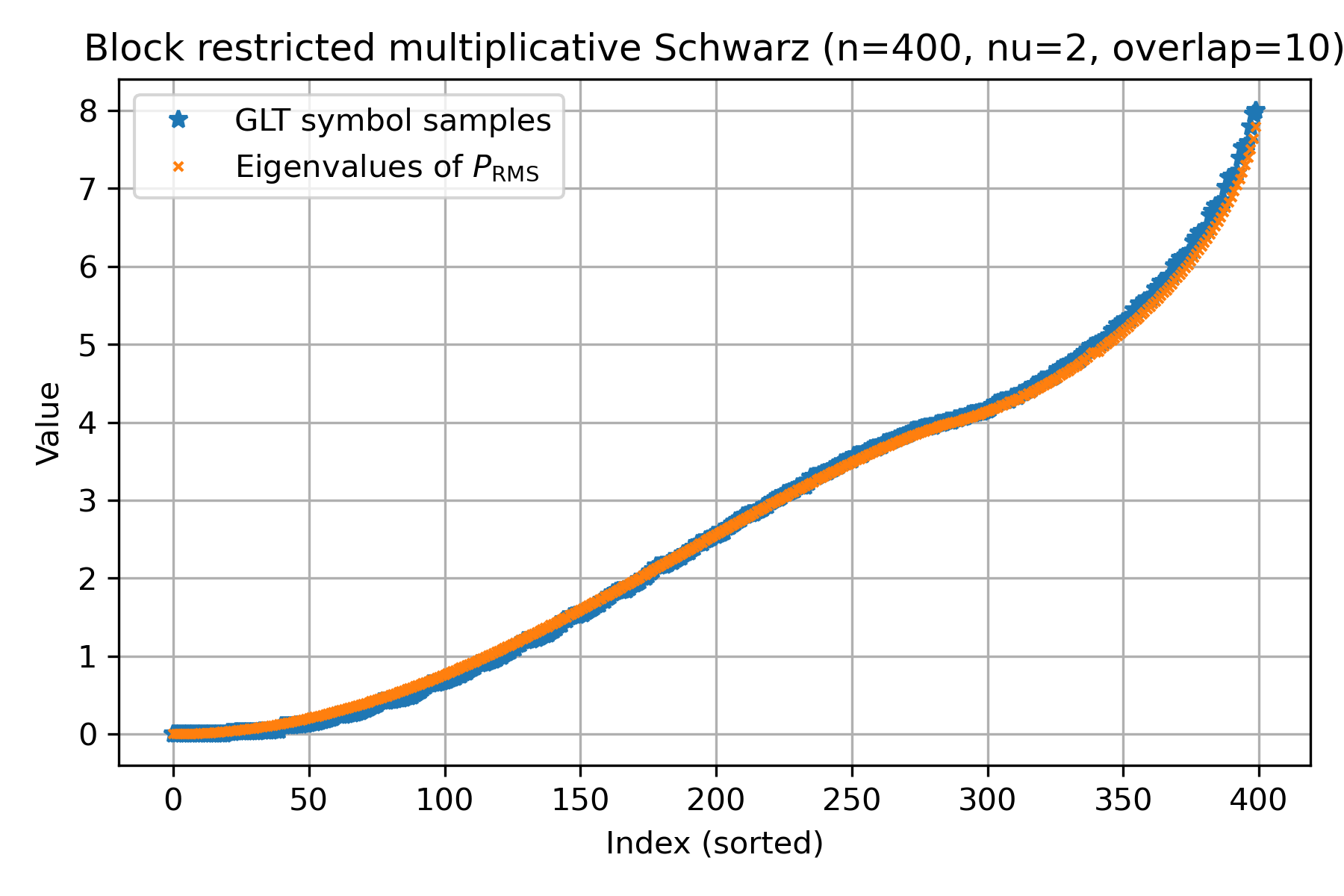}
    \includegraphics[width=0.49\textwidth]{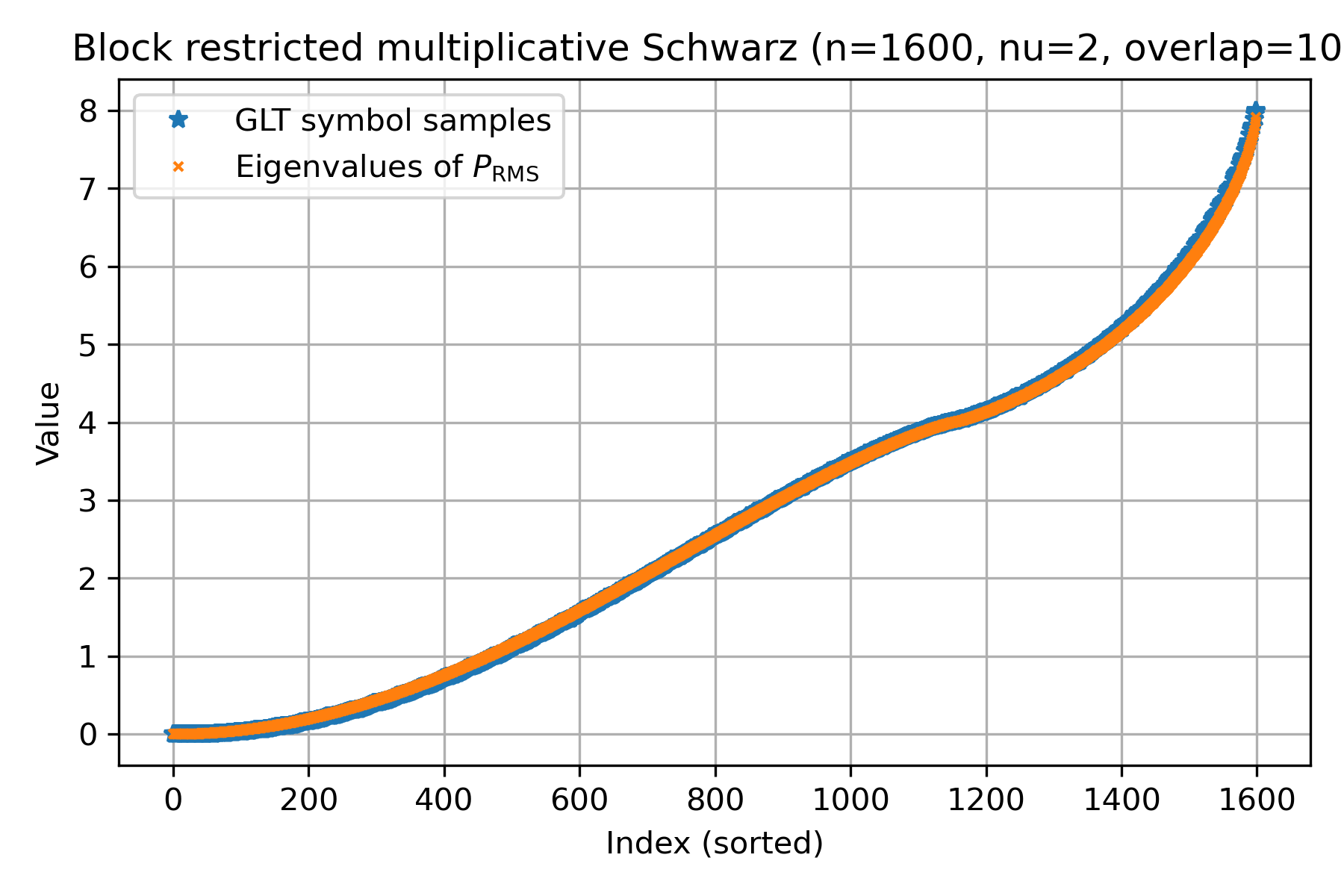}
    \caption{Comparison between the GLT symbol \((1+x^2)\,(2 - 2 \cos \theta)\) and the eigenvalues of the block restricted multiplicative Schwarz preconditioner \(P_{\mathrm{BRMS},n}^{(\nu)}\) with overlap $o=10$.
Left: $n=400$, $\nu=2$. Right: $n=1600$, $\nu=2$.}

    \label{Fig: FE 11}
     \end{figure}

In this case also the number $s$ in item~(A) is given by $s = 1$, we follow the program in item~(B). The results are collected in Tables~\eqref{tab:FEM1}--~\eqref{tab:FEM4}.

\begin{table}[H]
\centering
\resizebox{0.7\textwidth}{!}{%
\begin{tabular}{llccccccc}
\toprule
\textbf{Method} & $\nu$ & $n=40$ & $n=80$ & $n=160$ & $n=320$ & $n=640$ & $n=1280$ & $n=2560$ \\
\midrule

PCG & 2 & 4 & 4 & 4 & 4 & 4 & 4 & 4 \\
 & 4 & 8 & 8 & 8 & 8 & 8 & 8 & 8 \\
 & 8 & 10 & 14 & 16 & 16 & 16 & 17 & 17 \\
 & 16 & nac & 15 & 22 & 28 & 32 & 33 & 33 \\
\midrule

PGMRES & 2 & 4 & 4 & 4 & 4 & 4 & 4 & 4 \\
 & 4 & 8 & 8 & 8 & 8 & 8 & 8 & 13 \\
 & 8 & 15 & 14 & 16 & 23 & 16 & 28 & 17 \\
 & 16 & nac & 17 & 34 & 60 & 120 & 140 & 143 \\
\bottomrule
\end{tabular}}
\caption{Number of iterations for PCG and PGMRES with $P_{BAS, n}^{(\nu )}(K_n)$, and overlap $o=5$.}
\label{tab:FEM1}
\end{table}

\begin{table}[H]
\centering
\resizebox{0.8\textwidth}{!}{%
\begin{tabular}{llccccccc}
\toprule
\textbf{Method} & $\nu$ & $n=40$ & $n=80$ & $n=160$ & $n=320$ & $n=640$ & $n=1280$ & $n=2560$ \\
\midrule

\multirow{4}{*}{PGMRES} & 2 & 2 & 2 & 2 & 2 & 2 & 2 & 3 \\
 & 4 & 4 & 4 & 4 & 4 & 4 & 7 & 8 \\
 & 8 & 5 & 8 & 8 & 8 & 8 & 16 & 9 \\
 & 16 & nac & 12 & 16 & 16 & 16 & 17 & 21 \\
\bottomrule
\end{tabular}}
\caption{Number of iterations for PGMRES with $P_{BMS, n}^{(\nu )}(K_n)$, and overlap $o=5$.}
\label{tab:FEM2}
\end{table}

\begin{table}[H]
\centering
\resizebox{0.8\textwidth}{!}{%
\begin{tabular}{llccccccc}
\toprule
\textbf{Method} & \textbf{Metric} & $n=40$ & $n=80$ & $n=160$ & $n=320$ & $n=640$ & $n=1280$ & $n=2560$ \\

\midrule

\multirow{4}{*}{PGMRES}  & 2 & 3 & 3 & 3 & 3 & 3 & 3 & 3 \\
 & 4 & 7 & 7 & 7 & 7 & 7 & 7 & 14 \\
 & 8 & 16 & 13 & 15 & 15 & 15 & 30 & 16 \\
 & 16 & nac & 29 & 23 & 63 & 104 & 146 & 146 \\

\bottomrule
\end{tabular}}
\caption{Number of iterations for PGMRES with $P_{BRAS, n}^{(\nu )}(K_n)$, and overlap $o=5$.}
\label{tab:FEM3}
\end{table}

\begin{table}[H]
\centering
\resizebox{0.8\textwidth}{!}{%
\begin{tabular}{llccccccc}
\toprule
\textbf{Method} & \textbf{Metric} & $n=40$ & $n=80$ & $n=160$ & $n=320$ & $n=640$ & $n=1280$ & $n=2560$ \\
\midrule

\multirow{4}{*}{PGMRES} & 2 & 2 & 2 & 2 & 2 & 2 & 2 & 4 \\
 & 4 & 4 & 4 & 4 & 4 & 4 & 7 & 8 \\
 & 8 & 10 & 8 & 8 & 8 & 11 & 16 & 9 \\
 & 16 & nac & 17 & 16 & 16 & 18 & 17 & 31 \\

\bottomrule
\end{tabular}}
\caption{Number of iterations for PGMRES with $P_{BRMS, n}^{(\nu )}(K_n)$, and overlap $o=5$.}
\label{tab:FEM4}
\end{table}

\end{example}

\section{Conclusions}\label{sec7:end}

Our main results establish that for any GLT sequence \(\{A_n\}_n \sim_{\mathrm{GLT}} \kappa\), any block additive or multiplicative preconditioning matrix sequence  \(\{P_n\}_n\) also forms a GLT sequence with the same symbol \(\kappa\), with or without overlap and provided that the number of blocks \(\nu\) remains fixed independently of \(n\). 
Numerical results and visualizations are reported confirming the theoretical findings. A corresponding theoretical study for the preconditioned matrix sequences is given in twin paper \cite{twin}, together with new directions stemming also from the results of the present work.

\end{document}